\documentclass[11pt,a4paper]{amsart}

\usepackage{amsmath,amsfonts,amssymb,setspace}
\usepackage{enumerate}
\usepackage{todonotes}
\usepackage[left]{showlabels}
\usepackage[
      colorlinks=true,    
      urlcolor=blue,    
      menucolor=blue,    
      linkcolor=blue,    
      bookmarks=true,    
      bookmarksopen=true,
      citecolor=blue,    
      hyperfootnotes=true,    
      pdfpagemode=UseOutlines    
]{hyperref}
\usepackage[margin=2.5cm]{geometry}
\onehalfspace

\newtheorem{theorem}{Theorem}[section]
\newtheorem*{theorem*}{Theorem}
\newtheorem{lemma}[theorem]{Lemma}
\newtheorem{proposition}[theorem]{Proposition}

\newtheorem{corollary}[theorem]{Corollary}

\theoremstyle{definition}
\newtheorem{defn}[theorem]{Definition}
\newtheorem{ese}[theorem]{Example}

\numberwithin{equation}{section}
\numberwithin{theorem}{section}

\newcommand{\R}{\mathbb{R}}
\newcommand{\Z}{\mathbb{Z}}

\newcommand{\C}{\mathbb{C}}

\newcommand{\PP}{\mathbb{P}}
\newcommand{\Stab}{\mathrm{Stab}}

\newcommand{\Gal}{\text{Gal}}
\newcommand{\Mon}{\text{Mon}}

\newcommand{\mc}[1]{\mathcal{#1}}

\newcommand{\id}{\mathrm{id}}

\newcounter{mycomment}

\DeclareMathOperator{\rk}{\mathrm{rk}}

\author[F. Barroero]{Fabrizio Barroero}
\address{Università degli Studi Roma Tre\\
Largo San Murialdo 1, 00146, Roma\\
Italy}
\email{fbarroero@gmail.com}

\author[L. Capuano]{Laura Capuano}
\address{Dipartimento di Matematica ``L. Lagrange''\\
Politecnico di Torino \\
C.so Duca degli Abruzzi 24, 10129 Torino\\
Italy}
\email{laura.capuano1987@gmail.com}

\author[U. Zannier]{Umberto Zannier}
\address{Scuola Normale Superiore\\
Piazza dei Cavalieri 7, 56126 Pisa\\
Italy}
\email{u.zannier@sns.it}

\title{Betti maps, Pell equations in polynomials and almost-Belyi maps}
\date{\today}

\begin{document}

\begin{abstract}
	We study the Betti map of a particular (but relevant) section of the family of Jacobians of hyperelliptic curves using the polynomial Pell equation $A^2-DB^2=1$, with $A,B,D\in \C[t]$ and certain ramified covers $\PP^1\to \PP^1$ arising from such equation and having heavy constrains on their ramification. In particular, we obtain a special case of a result of André, Corvaja and Zannier on the submersivity of the Betti map by studying the locus of the polynomials $D$ that fit in a Pell equation inside the space of polynomials of fixed even degree. Moreover, Riemann Existence Theorem associates to the above-mentioned covers certain permutation representations: we are able to characterize the representations corresponding to ``primitive'' solutions of the Pell equation or to powers of solutions of lower degree and give a combinatorial description of these representations when $D$ has degree 4. In turn, this characterization gives back some precise information about the rational values of the Betti map.
\end{abstract}

\maketitle
\section{Introduction} \label{intro}

%
%

In the last few years Betti maps associated to sections of abelian schemes have been extensively studied and applied to problems of diophantine nature. In this article we present an apparently new approach, already sketched by the third author in \cite{Zannier_indam}, to study Betti maps in a special case, which is strictly related to the polynomial Pell equation. 
In particular, we explain how some properties of certain ramified covers of the projective line with prescribed ramification and arising from such equations can be used to prove a special case of a result by Andr\'e, Corvaja and the third named author \cite{ACZ18} on the distribution of the rational values of the Betti map given by a particular section of the family of Jacobians of hyperelliptic curves. In this special case our approach actually gives something more precise than the result in \cite{ACZ18}, and allows in principle to obtain a description of the preimage of the set of rational points with fixed denominator, see Corollary \ref{Coroll} and the subsequent discussion.

\medskip



Given an abelian scheme $\mc{A}\rightarrow S$ of relative dimension $g\geq 1$ over a smooth complex algebraic variety $S$ and a section $\sigma:S\rightarrow \mc{A}$, a Betti map for $\mc{A}$ and $\sigma$ is a real analytic map $\tilde{\beta}: \tilde{S}\rightarrow \R^{2g}$, where $\tilde{S}$ is the universal covering of $S(\C)$. This map  is of particular relevance in many different contexts; for example, rational points in the image of $\tilde{\beta} $ correspond to torsion values for the section, thus linking the Betti map to other diophantine problems. 
For an account about the use and the study of Betti maps see \cite[Section 1]{ACZ18}.

The rank $\rk \tilde{\beta}$ of the Betti map, namely the maximal value of the rank of the derivative $d\tilde{\beta}(\tilde{s})$ when $\tilde{s}$ runs through $\tilde{S}$, is of particular interest. Indeed, its maximality is equivalent to the fact that the image of the map contains a dense open subset of $\R^{2g}$, and implies that the preimage of the set of torsion points of $\mc{A}$ via $\sigma$ is dense in $S(\C)$, see \cite[2.1.1 Proposition]{ACZ18}.

In \cite{CMZ} the authors studied the rank of Betti maps associated to abelian surface schemes in the context of a relative Manin-Mumford problem. This work initiated more general investigations that led to \cite{ACZ18}, in which the authors conjectured a sufficient condition for the maximality of $\rk\tilde{\beta}$, proving it under some quite general and natural hypotheses (see also \cite{G182} for further results on this topic). 

In particular, in \cite{ACZ18} the authors handled the case of a specific non-torsion section of the Jacobian of the universal hyperelliptic curve of genus $g>0$. This case is relevant by itself but it is also linked to an issue raised by Serre \cite{Serre_Bourbaki}, who noticed a gap in an article by Robinson \cite{Robinson}, in which the family of hyperelliptic curves over the real numbers appears in connection with the Pell equation in polynomials. In \cite[Section 9]{ACZ18} the authors give an argument fixing this gap using Betti maps, see also \cite{Bog99,Peherstorfer,Totik,Lawrence16} for independent proofs of the same result.

 In the first part of this paper, we give an alternative proof of this result in the complex case for a particular but significant section. We obtain this as a consequence of the characterization of the dimension of particular subvarieties arising from solvable Pell equations in the ``moduli space'' of polynomials of even degree $\geq 4$.

Let $\mc{B}_{2d}$ be the Zariski-open subset of $\mathbb{A}^{2d}_\C$ defined by
\begin{equation}\label{polilocus}
\mc{B}_{2d}:=\{ (s_1,\dots, s_{2d} ) \in \mathbb{A}^{2d}_\C: \text{ the discriminant of } t^{2d}+s_{1}t^{2d-1}+\dots +s_{2d} \text{ is non-zero}  \}.
\end{equation}
Given a point $\overline{s}=(s_1,\dots, s_{2d} ) \in\mc{B}_{2d}$, we consider the affine hyperelliptic curve defined by $y^2=t^{2d}+s_{1}t^{2d-1}+\dots +s_{2d} $. If we homogenize this equation, we obtain a projective curve which is singular at infinity. There exists however a non-singular model $H_{\overline{s}}$ with two points at infinity which we denote by $\infty^+$ and $\infty^-$. We fix them by stipulating that the function $t^d\pm y$ has a zero at $\infty^{\pm}$. The curve $H_{\overline{s}}$ is then a hyperelliptic curve of genus $d-1$.

We denote by $\mc{J}_{\overline{s}}$ the Jacobian variety of $H_{\overline{s}}$; then, letting $\overline{s}$ vary in $\mc{B}_{2d}$, we have an abelian scheme $\mc{J}\rightarrow S:=\mc{B}_{2d}$ of relative dimension $g:=d-1$. We call $\sigma:S\rightarrow \mc{J}$ the section corresponding to the point $[\infty^+-\infty^-]$.

Let us give an informal definition of the Betti map associated to $\sigma$; for a precise definition we refer to \cite[Section 3]{ACZ18}.

In the constant case (i.e. when $\mc{J}$ is a complex abelian variety) an abelian logarithm of a point of $\mc{J}(\C)$ can be expressed as an $\R$-linear combination of the elements of a basis of the period lattice. These real coordinates are called \emph{Betti coordinates}. Note that they depend on the choice of the abelian logarithm and of the basis of the period lattice.

In the relative setting, the Betti map describes the variation of the Betti coordinates; locally (in the complex topology) on $S(\C)$, the Lie algebra $\text{Lie}(\mc{J})$ is the trivial vector bundle of rank $g$. Moreover, the kernel of the exponential map is a locally constant sheaf on $S(\C)$. One can then locally define a real analytic function that associates a point $\overline{s}$ of $S(\C)$ to the Betti coordinates of $\sigma(\overline{s})$.

One of the easiest cases which has been deeply studied in literature is the case of the Legendre scheme $\mathcal{L}$. In this case, the base $S$ is the curve $\PP^1 \setminus\{0,1,\infty\}$ and $\mathcal{L}\rightarrow \PP^1$ is the abelian scheme having fibers in $\PP_2$ defined by 
\[
ZY^2=X(X-Z)(X-\lambda Z),
\] 
for every $\lambda \in \PP_1 \setminus \{0,1,\infty\}$. Consequently, $\mathcal{L}$ is embedded in $\PP^1\setminus \{0,1,\infty\}\times \PP^2$. Locally, one can define a basis of periods using the hypergeometric functions. Indeed, for example for $\lambda$ in the region $D=\{|\lambda|<1 \mbox{ and } |1-\lambda|<1\}\subseteq S$, a suitable baisis of periods is given by $f(\lambda):=\pi F(\lambda)$ and $g(\lambda):=\pi iF(1-\lambda)$, where $F(\lambda)=\sum_{m=0}^{\infty} \frac{(m!)^2}{2^{4m}(m!)^4}\lambda^m$. Given a section $\sigma: S \rightarrow \mathcal L$, locally over $D$ one can take its elliptic logarithm $z$ and write it as $z=u_1f+u_2g$, where the functions $u_i$ are real analytic functions. Hence, the Betti map $\beta: D \rightarrow \R^2$ associated to $\sigma$ is given by $\beta(\lambda)=(u_1(\lambda), u_2(\lambda))$. 

In general, usually this map cannot be extended to the whole $S(\C)$ because of monodromy, but we can pass to the universal covering $\tilde{S}$ of $S(\C)$ and define a \emph{Betti map} $\tilde{\beta}: \tilde{S}\rightarrow \R^{2(d-1)}$, which is not unique and depends of several choices. 
The following is our first result.

\begin{theorem}\label{Thm:Betti}
Let $d$ be an integer $\geq 2$ and let $\mc{J}\rightarrow S$ and $\sigma:S\rightarrow \mc{J}$ be the abelian scheme and the section defined above. Then, the corresponding Betti map $\tilde{\beta}: \tilde{S}\rightarrow \R^{2(d-1)}$ satisfies $\rk \tilde{\beta}\geq 2(d-1)$;  equivalently, $\tilde{\beta}$ is submersive on a dense open subset of $\tilde{S}$. In particular, the set of $\overline{s}\in S(\C)$ such that $\sigma(\overline{s})$ is torsion on $\mc{J}_{\overline{s}}$ is dense in $S(\C)$ in the complex topology.
\end{theorem}

The above theorem is a special case of \cite[2.3.3 Theorem]{ACZ18}; however, we are going to give a different proof of it in Section \ref{Proof of 1.1}.
Our proof makes use of the connection between the torsion values of the section $\sigma: S \rightarrow \mc{J}$ corresponding to the point $[\infty^+ - \infty^-]$ and the polynomial Pell equation.
\medskip

We recall that the classical Pell equation is an equation of the form $A^2-DB^2=1$ where $D$ is a positive integer, to be solved in integers $A$ and $B$ with $B\neq 0$. A theorem of Lagrange says that such an equation is non-trivially solvable if and only if $D$ is not a perfect square.

We consider the polynomial analogue of this problem, replacing $\Z$ by a polynomial ring over a field. This variant is old as well, and can be dated back to studies by Abel \cite{Abel} especially in the context of integration in finite terms of certain algebraic differentials. Lately, these studies have been carried out by several authors (see \cite{Bogatirev, Serre_Bourbaki, MZpreprint} for more details). Apart from their link with points of finite order in Jacobians of hyperelliptic curves, the polynomial Pell equation, which from now on we will call Pell-Abel equation as suggested by Serre in \cite{Serre_Bourbaki}, has several connections with other mathematical problems, like polynomial continued fractions \cite{Zan19}, elliptical billiards \cite{DR11, CZbilliards}, but also problems in mathematical physics \cite{BZ13} and dynamical systems \cite{McM06}.

Let $K$ be a field of characteristic 0. For a non-constant polynomial $D(t)\in K[t]$, we look for solutions of
\begin{equation} \label{Pell}
	A(t)^2-D(t)B(t)^2=1, 
\end{equation}
where $A(t),B(t)\in K[t]$ and $B\neq 0$; if such a solution exists, we call the polynomial $D$ \emph{Pellian}.
Clear necessary conditions for a polynomial $D$ to be Pellian
are that $D$ has even degree, it is not a square in $K[t]$ but the leading coefficient is a square in $K$. Unlike the integer case, these conditions are not sufficient to guarantee the existence of a non-trivial solution, and there are examples of non-Pellian polynomials satisfying these conditions (see for example \cite{Zannier_indam}).

%

If $K$ is algebraically closed (in particular $K=\C$) there is a criterion, attributed to Chebyshev in \cite{Berry05}, that links the solvability of a Pell-Abel equation to the order of the point $[\infty^+-\infty^-]$ in the Jacobian of the hyperelliptic curve defined by the equation $y^2=D(t)$. 

\begin{proposition}[{see \cite[Proposition 12.1]{Zannier_indam}}] \label{Prop.Abel}
	Let $D\in K[t]$ be a squarefree polynomial of degree $2d\geq 4$. Then, the Pell-Abel equation $A^2-DB^2=1$ has a non-trivial solution $A,B \in \overline{K}[t]$ with $B \neq 0$ if and only if the point $[\infty^+-\infty^-]$ has finite order in the Jacobian of the smooth projective model of the affine hyperelliptic curve of equation $y^2=D(t)$. Moreover, the order of $[\infty^+-\infty^-]$ is the minimal degree of the polynomial $A$ of a non-trivial solution.
\end{proposition}

We point out that a similar criterion holds as well in the case of non-squarefree $D$, by using generalized Jacobians (for more on this, see \cite{Zan19}).

As in the integer case, if $D(t)$ is Pellian, then the associated Pell-Abel equation has infinitely many solutions in $\overline{K}[t]$.
Indeed, a possible non-trivial solution $(A,B)$ generates infinitely many ones by taking powers $A_m+\sqrt{D}B_m:=\pm \left(A\pm \sqrt{D}B\right)^m$. Moreover, solutions to the Pell-Abel equation correspond to the units of the ring $\overline{K}[\sqrt{D}]$ which form a group isomorphic to $\Z \oplus \Z/2\Z$.

We will call a solution \emph{primitive} if it has minimal degree among all the non-trivial ones. On the other hand, we say that a solution $(A_m,B_m)$ is an \emph{$m$-th power} when it can be obtained from another one as explained above. Notice that every solution of the Pell-Abel equation will then be a power of a primitive one.

We call \emph{degree of a solution} $(A,B)$ the degree of $A$. Thus, the order of $[\infty^+-\infty^-]$, if finite, is the degree of a primitive solution of the corresponding Pell-Abel equation.

If we consider the abelian scheme $\mc{J} \rightarrow S=\mc{B}_{2d}$ as above and the section $\sigma: S\rightarrow \mc{J}$ corresponding to the point $[\infty^+-\infty^-]$, we have that the polynomial $D_{\overline{s}}=t^{2d}+s_{1}t^{2d-1}+\dots +s_{2d}$ is Pellian with primitive solution of degree $n$ if and only if $\sigma(\overline{s})$ is a torsion point of order $n$.


We define the \textit{Pellian locus}  $\mc{P}_{2d}$ in $\mc{B}_{2d}$ to be
\begin{equation}\label{pellianlocus}
\mc{P}_{2d}:=\{ \overline{s}=(s_1,\dots, s_{2d} ) \in \mc{B}_{2d}: \text{ the polynomial $D_{\overline{s}}=t^{2d}+s_{1}t^{2d-1}+\dots +s_{2d}$ is Pellian}  \}.
\end{equation}

We have the following result.

\begin{theorem}\label{Thm.dens} For $d\geq 2$, the set $\mc{P}_{2d}(\C)$ of Pellian complex polynomials is dense in $\mc{B}_{2d}(\C)$ with respect to the complex topology. 
\end{theorem}

We point out here that the distribution of Pellian polynomials in families and its connection to problems of unlikely intersections has been investigated (often using Betti maps) in recent years, mainly in the case of curves in $\mc{B}_{2d}$ where the behaviour is completely opposite to what happens in the setting of the above theorem (for a general account on the problems of unlikely intersections, see \cite{Zannier}). Indeed, for instance, in \cite{MZ14b} and \cite{MZpreprint}, the authors show that a generic curve in $\mc{B}_{2d}$, where $d\geq 3$, contains at most finitely many complex points that correspond to Pellian polynomials. One can see also \cite{BMPZ} and \cite{Harry} for the non-squarefree case and \cite{BC20} for similar results for the generalized Pell-Abel equation.


Clearly, Theorem \ref{Thm.dens} is a consequence of Proposition \ref{Prop.Abel} and Theorem \ref{Thm:Betti} (and thus of \cite{ACZ18}); however, in this paper we present a different proof of this two results with the principal aim of showing the link between the Betti maps and some properties of certain ramified covers of the projective line with fixed ramification. Moreover, we are able to study more deeply the Pellian locus $\mc{P}_{2d} \subseteq \mc{B}_{2d}$ and showing that it consists of a denumerable union of algebraic subvarieties of $\mc{B}_{2d}$ of dimension at most $d+1$ (see Proposition \ref{Prop.dimension}), each one coming from the preimage of a rational point of a Betti map $\tilde{\beta} $. If $\tilde{\beta}$ did not have maximal rank, then these preimages would have dimension strictly larger than $d+1$. 

To study the Pellian locus $\mc{P}_{2d}$, we associate to a degree $n$ solution $(A,B)$ of a Pell-Abel equation the ramified cover of degree $2n$ given by $A^2:\mathbb{P}^1(\C)\rightarrow \mathbb{P}^1(\C)$. Because of the Pell-Abel equation, the latter has to be ramified above 0, 1, and $\infty$ and, by the Riemann-Hurwitz formula, at most $d-1$ additional points, independently of $n$. For this reason we call these maps, following the third author in \cite{Zannier_indam}, \emph{almost-Belyi maps}.

Riemann Existence Theorem gives a link between such covers and permutation representations. To be more precise, if our cover of degree $2n$ has $h$ branch points, an $h$-tuple $\Sigma$ of elements of $S_{2n}$ can be associated to it. These permutations satisfy certain properties, e.g., they have very strong constrains on their cycle structure. On the other hand, fixing the branch points and a tuple of permutations satisfying these properties determines the map up to automorphisms of the domain.

It turns out that one can determine whether a solution associated to a certain representation is primitive or a power of another one by looking at the partitions of $\{1, \dots , 2n\}$ preserved by the subgroup of $S_{2n}$ generated by the same representation, i.e., the \textit{monodromy group} of the cover.

\begin{theorem}[Theorem \ref{characterization}]
	Let $(A,B)$ be a solution of degree $n$ of a Pell-Abel equation $A^2-DB^2=1$ with $\deg D=2d$ and let $G_A$ be the monodromy group of $A^2$. Then, for every integer $m\mid n$ with $\frac{n}{m} \ge d$, $(A,B)$ is the $m$-th power of another solution if and only if $G_A$ preserves a partition of  the set $\{1,\dots, 2n\}$ in $2m$ subsets, each of cardinality $\frac{n}{m}$.
\end{theorem}

In \cite{Zannier_indam}, the third author gives an example of a permutation representation associated to a Pell-Abel equation of arbitrary degree, see Example \ref{Ex.Zannier} below. Using the above criterion it is easy to see that the corresponding solution of the Pell-Abel equation is primitive and this implies that, after fixing their degree, there are Pellian polynomials with primitive solutions of any possible degree.

\begin{corollary}\label{Thm.prim}
Let $d,n$ be positive integers with $n\geq d\geq 2$. Then, there exists a squarefree Pellian $D\in \C[t]$ of degree $2d$ such that a primitive solution of the corresponding Pell-Abel equation \eqref{Pell} has degree $n$.
\end{corollary}



Going back to our Betti map $\tilde{\beta}: \tilde{S}\rightarrow \R^{2(d-1)}$ associated to the abelian scheme $\mc{J} \rightarrow S$, we can deduce some properties of the rational points in the image $\tilde{\beta}(\tilde{S})$. We call denominator of a rational point of $\R^{2(d-1)}$ the positive least common denominator of the coordinates of the point. 

It is easy to see that there are no non-zero points of denominator $< d$, since a non-trivial solution of the Pell-Abel equation has degree at least $d$. From \cite{ACZ18}, one can easily deduce that $\tilde{\beta}(\tilde{S})$ contains rational points of every denominator $\geq n_0$ for some large enough natural number $n_0$. Our Corollary \ref{Thm.prim} implies the following result.

\begin{corollary}\label{Coroll}
The image $\tilde{\beta}(\tilde{S})$ of the Betti map contains a point of denominator $n$ for all $n\geq d$.
\end{corollary}

In a recent work \cite{Corvaja_Demeio_Masser_Zannier}, the authors consider the distribution of points of the base of an elliptic scheme where a section takes torsion values. Those are exactly the points of the base where the corresponding Betti map takes rational values. More specifically, they prove that the number of points where the Betti map takes a rational value with denominator (dividing) $n$ is, for large $n$, essentially $n^2$ times the area of the base with respect to the measure obtained locally by pulling back the Lebesgue measure on $\R^2$ by the Betti map. Moreover, they show that the constant of the main term is equal to $\hat{h}(\sigma)$, where $\hat h$ is the canonical height associated to twice the divisor at infinity on the elliptic scheme.

The corresponding result for our family of Jacobians $\mc{J} \rightarrow S$ and our section would probably give an asymptotic formula for the number of equivalence classes of polynomials $D$ of fixed degree $2d\geq 4$ satisfying a Pell-Abel equation $A^2-DB^2=1$ with $A$ of degree $n$, for $n$ tending to infinity.

On the other hand, using the correspondence given by Riemann Existence Theorem explained above, one is able to determine the exact number of equivalence classes of Pell-Abel equations of fixed degree $n$. In Section \ref{sec:d=2}, we compute this number in the case $d=2$. We point out that, as $d$ grows, the combinatorics behind the problem becomes more complicated, but this approach would still work in principle.

In turn, this complete classification should allow to compute the number of components of $\mc{P}_{2d}\subseteq \mc{B}_{2d}$ that map via the Betti map to rational points of fixed denominator rather than just an asymptotic formula analogous to the one in \cite{Corvaja_Demeio_Masser_Zannier}.

\section*{Acknowledgements}
The three authors started this work in Pisa, where they were funded by the European Research Council, grant number 267273. The first author was also founded by the Engineering and Physical Sciences Research Council [EP/N007956/1] and by the Swiss National Science Foundation [165525]. The second author was also founded by Istituto Nazionale di Alta Matematica [Borsa Ing. G. Schirillo], and by the Engineering and Physical Sciences Research Council [EP/N008359/1].

\section{Almost-Belyi maps}\label{ABm}

Given a (squarefree) Pellian polynomial $D(t)$ of degree $2d$ and a solution $(A, B)$ of degree $n$ of the corresponding Pell-Abel equation $A^2-D(t)B^2=1$, we consider the map $\phi_A:=A^2 : \PP^1(\C) \rightarrow \PP^1(\C)$ of degree $2n$.

It is easy to see that the Pell-Abel equation forces the branching of the map $\phi_A$ to be ``concentrated'' in $ 0,\ 1$ and $\infty$. First, since $\phi_A$ is a polynomial, there is total ramification above $\infty$. Moreover, as $\phi_{A}$ is a square and $\phi_A-1=DB^2$, we have that the ramification indices above $0$ are all even, while the ones above $1$ are all even with the exception of $2d$ points (the simple roots of $D(t)$). Hence, counting the branching of $\phi_A$ as the sum of $e-1$ over the ramification indices $e$, we have that, above $0$ the branching is at least $n$, above $1$ at least $n-d$ and above $\infty$ it is exactly $2n-1$. On the other hand, by the Riemann-Hurwitz formula (see \cite[Theorem A.4.2.5., p.~72]{Hindry_Silverman}), the total branching is equal to $4n-2$. This means that the branching outside $0,1,\infty$ is at most $d-1$, and thus there cannot be more than $d-1$ further branch points. Following the third author \cite{Zan19}, we call $\phi_A$ an \textit{almost-Belyi map} as its branching is ``concentrated'' above $0,1, \infty$ (in the sense that the number of branch points outside this set does not depend on the degree $n$ of $\phi_A$ but only on the degree of $D$ which is fixed). 

To each such map, we can associate a monodromy permutation representation in the following way (for references see \cite{Fried77}, \cite{Miranda} or \cite{Volklein}).

Let us call $\mathcal{B}=\{b_1, \ldots, b_{h}\}$ the set of the branch points of  $\phi_A$ (where $h\le d+2$) and let us choose a base point $q$ different from the $b_i$. Let us moreover call $V:=\PP^1(\C)\setminus \mathcal{B}$. The fundamental group $\pi_1(V,q)$ of $V$ is a free group on $h$ generators $[\gamma_1],\ldots,[\gamma_{h}]$ modulo the relation
\[ 
[\gamma_1]\cdots[\gamma_{h}]=1, 
\]
where each $\gamma_i:[0,1]\rightarrow V$ is a closed path which winds once around $b_i$.
Now, consider the fiber $\phi_A^{-1}(q)$ above $q$, and denote by $q_1, \ldots, q_{2n}$ the $2n$ distinct points in this fiber. Every loop $\gamma \subseteq V$ based at $q$ and not passing through the $b_i$ can be lifted to $2n$ paths $\tilde{\gamma}_1, \ldots, \tilde{\gamma}_{2n}$, where $\tilde{\gamma}_j$ is the unique lift of $\gamma$ starting at $q_j$. Hence, $\tilde{\gamma}_j(0)=q_j$ for every $j$. Now, consider the endpoints $\tilde{\gamma}_j(1)$; these also lie in the fiber $\phi_A^{-1}(q)$, and indeed form the entire preimage set $\{q_1, \ldots, q_{2n}\}$. We call $\sigma(j)$ the index in $\{1, \ldots, 2n\}$ such that $\tilde{\gamma}_j(1)=q_{\sigma(j)}$. 
The function $\sigma$ is then a permutation of the set $\{1, \ldots, 2n\}$ which depends only on the homotopy class of $\gamma$; therefore, we have a group homomorphism $\rho:\pi_1(V,q) \longrightarrow S_{2n}$ called the monodromy representation of the covering map $\phi_A$. This is clearly determined by the images $\sigma_i=\rho([\gamma_i])$ of the generators of $\pi_1(V,q)$, and the image of $\rho$ must be a transitive subgroup of $S_{2n}$, as $V$ is connected.

Note that these permutations must have a precise cicle structure depending on the branch points of the map. Following the notation of \cite{Bilu99}, we say that a branch point $c$ of a rational map $f$ is of type $(m_1, \ldots, m_k)$ if, for $ f^{-1}(c)=\{c_1, \dots, c_k\}$, the point $c_j$ has ramification index $m_j$ for every $j=1, \ldots, k$. The type of a branch point corresponds to the cycle structure of the corresponding permutation, meaning that, going back to $\phi_A$, if $b_i$ is of type $(m_1, \ldots, m_k)$, then $\sigma_i$ is the product of $k$ disjoint cycles $\tau_1,\dots,  \tau_k$ with $\tau_j$ of length $m_j$.

We just showed how to associate to a covering a certain tuple of permutations satisfying some specific decomposition properties; Riemann Existence Theorem tells us that we can do the opposite, i.e., to any tuple of permutations satisfying certain properties we can associate a covering. 
One can actually say more: there exists a 1-1 correspondence between these two sets, if one mods out by the following equivalence relations on both sides.

\begin{defn}
Two rational maps $F_1,F_2:\PP^1(\C) \longrightarrow \PP^1(\C)$ are called \textit{equivalent} if there is an automorphism $\varphi$ of $\PP^1(\C)$ such that $F_1=F_2\circ \varphi$.
\end{defn}

\begin{defn}
If $\Sigma=(\sigma_1, \sigma_2, \cdots ,\sigma_{h})$ and $\Sigma'=(\sigma'_1, \sigma'_2, \cdots, \sigma'_{h})$ are two $h$-tuples of permutations of $S_{2n}$, we say that they are \textit{conjugated} if there exists a $\tau\in S_{2n}$ such that $\tau^{-1}\sigma_i\tau=\sigma_i'$ for every $i=1, \ldots, h$. If this holds, we use the notation $\tau^{-1}\Sigma \tau=\Sigma'$.
\end{defn}

We can finally state this consequence of Riemann Existence Theorem.

\begin{theorem}[\cite{Miranda}, Corollary 4.10] \label{thm_miranda}
Fix a finite set $\mathcal B=\{b_1, \ldots, b_h\}\subseteq \PP^1(\C)$. Then, there is a 1-1 correspondence between
\[
\begin{Bmatrix}
\text{equivalence classes}\\ \text{ of rational maps} \\
F:\PP^1(\C) \longrightarrow \PP^1(\C) \\
\text{of degree }2n \\
\text{whose branch points lie in }\mathcal{B}
\end{Bmatrix}  \text{ and } 
\begin{Bmatrix}
\text{conjugacy classes } 
(\sigma_1, \ldots, \sigma_{h})\in S_{2n}^h  \\
\text{such that }\sigma_1\cdots \sigma_{h}=\mathrm{id},\\
\text{the subgroup generated by the } \sigma_i \\
\text{is transitive and } \sum_{i,j}(m_{ij}-1)=4n-2, \\
\text{where } (m_{i1},\dots, m_{ik_i}) \text{ is the cycle structure of } \sigma_i
\end{Bmatrix}. \]
Moreover, given the tuple of permutations $(\sigma_1, \ldots, \sigma_{h})$, for each $i=1,\dots, h$ there are $k_i$ preimages $b_{i1}, \ldots, b_{ik_i}$ of $b_i$ for the corresponding cover $F: \PP^1(\C) \longrightarrow \PP^1(\C) $, with $e_{F}(b_{ij})=m_{ij}$.
\end{theorem}

As seen before, given a Pell-Abel equation \eqref{Pell} with $A$ of degree $n$ and $D$ of degree $2d$, we get the map $\phi_A$ that is branched in $\{0,1,\infty\}$ and in further $k $ points, where $0\leq k \le d-1$. Call $\mathcal{B}:=\{0,1,\infty,b_1, \ldots, b_{k}\}$ and $V:=\PP_1(\C)\setminus \mathcal{B}$. We fix a $q\in V$ and we let $\rho:\pi_1(V,q) \longrightarrow S_{2n}$ be the associated monodromy representation (as explained above). We now let $\gamma_0,\gamma_{\infty},\gamma_{1}$ be closed paths winding once around $0,\infty,1$ respectively and $\delta_i$ be a closed path winding once around $b_i$ for all $i=1, \ldots, k$. Then, $[\gamma_0] [\gamma_{\infty}] [\gamma_{1}] [\delta_1] \cdots [\delta_{k}]=1$ and all these classes generate $\pi_1(V,q)$.

We let $\sigma_0=\rho(\gamma_0)$, $\sigma_{\infty}=\rho(\gamma_{\infty})$, $\sigma_1=\rho(\gamma_1)$ and $\tau_i=\rho(\delta_i)$ for $i=1, \ldots, k$.
 Then, we have that  
 \begin{enumerate}[1 -]
 	\item $\sigma_0\sigma_{\infty}\sigma_1\tau_1\cdots\tau_{k}=\mathrm{id}$;
 	\item the subgroup of $S_{2n}$ generated by these permutations is transitive;
 	\item $\sigma_0$ and $\sigma_{\infty}$ do not fix any index;
 	\item $\sigma_1$ fixes exactly $2d$ indexes.
 \end{enumerate} 

Moreover, concerning the decomposition of these permutations in disjoint cycles we have that

 \begin{enumerate}[1 -]
	\item all non-trivial cycles in $\sigma_0$ and $\sigma_1$ have even length;
	\item $\sigma_{\infty}$ must be a $2n$-cycle;
	\item the sum over all cycles of $\sigma_0, \dots , \tau_k$ of their lengths minus 1 must give $4n-2$.
\end{enumerate} 

All of this depends on the chosen labelling of the preimages of $q$, so a Pell-Abel equation gives exactly one conjugacy class of $(k+3)$-tuples and a $\mc{B}$ for some $k\in \{0,\dots , d-1\}$.

On the other hand, fixing $k\in \{0,\dots , d-1\}$ and $\mc{B}$,  a conjugacy class of $(k+3)$-tuples satisfying the above conditions gives a cover $\PP^1(\C)\rightarrow \PP^1(\C)$ that is a polynomial branched exactly in $\mc{B}$. Moreover, the cycle structures of $\sigma_0$ and $\sigma_1$ make this polynomial satisfy a Pell-Abel equation. 

We see an example.

\begin{ese}\label{Ex.Zannier}
	Let us consider the case in which the branching outside $0,1, \infty$ is maximal, i.e. $k=d-1$. In this case, $\sigma_1$ is the product of $n-d$ transpositions, $\sigma_0$ is the product of $2n$ transpositions and each $\tau_i$ consists of a single transposition. 
	
	A possible choice given by the third author in \cite{Zannier_indam} and leading to a Pell-Abel equation (\ref{Pell}), is $\sigma_{\infty}=(2n, \ldots, 1)$, $\sigma_0=(1,2n)(2,2n-1)\cdots (n,n+1)$, $\sigma_1=(1,2n-1)(2,2n-2)\cdots (n-d,n+d)$ and $\tau_i=(n-i,n+i)$, $i=1, \ldots, d-1$. We will prove in the next sections that this example corresponds to a primitive solution of a Pell-Abel equation.
\end{ese}

\section{Proof of Theorem \ref{Thm:Betti}} \label{Proof of 1.1}

This section is devoted to proving Theorem \ref{Thm:Betti}.
 For the reader's convenience, we recall the notations and the statement of the theorem. 
Given a point $\overline{s}=(s_1, \ldots, s_{2d})\in \mc{B}_{2d}$, we denote by $H_{\overline s}$ a non-singular model of the hyperelliptic curve defined by $y^2=t^{2d}+ \cdots +s_{2d}$ with two points at infinity that we call $\infty^+$ and $\infty^-$, and by $\mathcal{J}_{\overline s}$ its Jacobian variety. We denote by $\mathcal{J} \rightarrow S=\mc{B}_{2d}$ the abelian scheme of relative dimension $d-1$ whose fibers are the $\mathcal{J}_{s}$ and by $\sigma:S \rightarrow \mathcal{J}$ the section corresponding to the point $[\infty^+-\infty^-]$.
We recall moreover that $\tilde{S}$ is the universal covering of $S(\C)$. 

\begin{theorem*}
Let $\tilde{\beta}: \tilde{S} \rightarrow \R^{2(d-1)}$ be the Betti map associated to the section $\sigma$; then, $\rk \tilde{\beta}\ge 2(d-1)$ and, equivalently, $\tilde{\beta}$ is submersive on a dense open subset of $\tilde S$. In particular, the set of $\overline{s}\in S(\C)$ such that $\sigma(\overline s)$ is torsion on $\mathcal{J}_{\overline s}$ is dense in $S(\C)$ with respect to the complex topology.
\end{theorem*}

\begin{proof}
	The rank of $\tilde{\beta}$ is trivially bounded from above by $2(d-1)$, so it is enough to show that there exists a point $\tilde{s}\in \tilde{S}$ where the rank is $\geq 2(d-1)$.

Recall that $S$ has dimension $2d$. Suppose by contradiction that the maximal rank is $ r< 2(d-1)$; then, the non-empty fibers of $\tilde{\beta}$ are complex analytic varieties of dimension $\ge 2d-r/2>d+1$, see \cite[Propositions 2.1 and 2.2]{CMZ}.

Now, by applying Example \ref{Ex.Zannier} and Theorem \ref{thm_miranda}, we have that the set of Pellian polynomials of degree $2d$ is not empty, and this gives the existence of a rational point, with denominator say $n$, in the image of $\tilde{\beta}$. We call $\tilde{V}$ an irreducible component of the preimage of that point;  we may assume it has dimension $>d+1$. Moreover, we let $V$ be its image in $S(\C)$; this must again have dimension $>d+1$ and we have that each point of $\sigma(V)$ is a torsion point of order $n$ in the respective fiber. 

Then, we have $V\subseteq \mc{P}_{2d}(\C)$, where $\mc{P}_{2d}$ is the Pellian locus defined in \eqref{pellianlocus} as, by Proposition \ref{Prop.Abel}, points of $V$ correspond to some Pellian polynomials with a solution of degree $n$. 

The following Proposition \ref{Prop.dimension} concludes the proof of Theorem \ref{Thm:Betti}, since it contradicts the lower bound $\dim V >d+1$.
\end{proof}

\begin{proposition}\label{Prop.dimension} 
The pellian locus $\mc{P}_{2d}$ consists of denumerably many algebraic subvarieties of $\mc{B}_{2d}$ of dimension at most $d+1$.
\end{proposition}

\begin{proof}
As the irreducible components of $\mc{P}_{2d}$ correspond to rational values of the Betti map, we clearly have that the number of components of $\mc{P}_{2d}$ is countable. Moreover, they are components of projections on $\mc{B}_{2d}$ of torsion subgroup schemes of $\mathcal{J}$ and therefore they are algebraic subvarieties of $\mc{B}_{2d}$.

Fix now a component $U$ of $\mc{P}_{2d}$ of dimension $>d+1$ mapping to a rational point of denominator $n$. Then, every point of $U$ corresponds to a Pellian polynomial with a solution of degree $n$.
We let $W$ be the closure of
\begin{multline*}
\{(a_0,\dots ,a_{2n})\in \mathbb{A}_\C^{2n+1}: a_{0}t^{2n}+a_{1}t^{2n-1}+\dots+ a_{2n} =A^2 \mbox{ is a square} \\ \text{and }  A^2-D_{\overline{s}}B^2=1 \mbox{ for some $\overline{s}\in U$} \}
\end{multline*}
in
$$
\mc{A}_{2n}:=\{(a_0,\dots ,a_{2n})\in \mathbb{A}_\C^{2n+1}: a_{0}\neq 0 \}.
$$
Note that $W$ is the closure of the projection on $\mc{A}_{2n}$ of
\begin{multline*}
	\{	(a_0,\dots ,a_{2n},b_0,\dots, b_{2n-2d},\overline{s})\in \mc{A}_{2n}\times \mathbb{A}_\C^{2n-2d+1}\times U: \\ a_{0}t^{2n}+\dots+ a_{2n} =A^2 \text{ and } b_{0}t^{2n-2d}+\dots+ b_{2n-2d} =B^2 \mbox{ are squares} \\ \text{and }  A^2-D_{\overline{s}}B^2=1 \}.
\end{multline*}
Such projection has finite fibers, therefore $W$ has dimension $>d+1$.


Note that it is possible to compose any degree $2n$ complex polynomial $f$ with a linear polynomial and obtain a monic polynomial with no term of degree $2n-1$.  Moreover, there are at most $2n$ possible polynomials of this special form that can be obtained in this way from a given $f$. We let
$$
\mc{C}_{2n}:=\{(c_2,\dots ,c_{2n})\in \mathbb{A}_\C^{2n-1} \} \text{ and } \mc{L}:=\{(a,b)\in \mathbb{A}_\C^2: a\neq 0\},
$$
and define the morphism $\varphi:\mc{C}_{2n}\times \mc{L}\rightarrow\mc{A}_{2n}$,
\begin{multline*}
\varphi((c_2,\dots ,c_{2n}),(a,b)) =(a_0,\dots ,a_{2n}) \iff \\ a_0t^{2n}+\dots +a_{2n}=(at+b)^{2n}+c_{2}(at+b)^{2n-2}+\cdots+ c_{2n-1}(at+b)+ c_{2n} .
\end{multline*}

By the above considerations, $\varphi$ is surjective and the fibers are finite of cardinality at most $2n$.

If we let $\pi$ be the projection $ \mc{C}_{2n}\times \mc{L}\to \mc{C}_{2n}$ and $Z=\pi(\varphi^{-1}(W))$, as the fibers of $\pi$ have dimension 2, $Z$ must have dimension $>d-1$.  

Now, for $\overline{c}=(c_2,\dots ,c_{2n})\in Z$, the corresponding polynomial $f_{\overline{c}}(t)=t^{2n}+c_{2}t^{2n-2}+\cdots+ c_{2n}$ is a square and we have $f_{\overline{c}}(t)-D(t)B(t)^2=1$ for some polynomials $D(t),B(t)$ with $D$ of degree $2d$. We let 
\begin{multline*}
\tilde{Z}=	\{	(\overline{c},c'_{2},\dots, c'_{2n-1},r_1,\dots, r_{2n-1})\in  Z \times  \mathbb{A}_\C^{2n-2}\times \mathbb{A}_\C^{2n-1} : f'_{\overline{c}}= 2n t^{2n-1}+ c'_{2} t^{2n-3}+\dots+ c'_{2n-1} \\ \text{and }  f'_{\overline{c}} = 2n(t^{2n-1}+ \sigma_2(r_1,\dots, r_{2n-1})t^{2n-3}+\dots+ \sigma_{2n-1}(r_1,\dots, r_{2n-1})) \},
\end{multline*}
where $\sigma_i(X_1,\dots, X_{2n-1})$ is the $i$-th elementary symmetric polynomial in $2n-1$ variables and $f'_{\overline{c}}$ is the derivative of $f_{\overline{c}}$. In other words, given $\overline c \in Z$, $c_2',\ldots, c_{2n-1}'$ are the coefficients of the derivative of the polynomial $f_{\overline c}(t)$ and $r_1,\ldots, r_{2n-1}$ are the ramification points of the same polynomial (not necessarily distinct).
Clearly, $\tilde{Z}$ must have dimension $>d-1$.

Finally, if we consider the morphism $\psi:\tilde{Z} \to \mathbb{A}_\C^{2n-1}  $ defined by $$(\overline{c},c'_{2},\dots, c'_{2n-1},r_1,\dots, r_{2n-1})\mapsto (f_{\overline{c}}(r_1), \dots , f_{\overline{c}}( r_{2n-1})),$$
then $f_{\overline{c}}(r_1), \dots , f_{\overline{c}}( r_{2n-1})$ are the (not necessarily distinct) branch points of $f_{\overline{c}}$.
By Riemann Existence Theorem and the above considerations, the fibers of $\psi$ are finite and have cardinality at most $2n$, and therefore $\psi(\tilde{Z})$ has dimension $>d-1$.

This gives a contradiction because the considerations of the previous section imply that the polynomial $f_{\overline{c}}$, since it fits in a Pell-Abel equation, may have not more than $d+1$ finite branch points, two of which are 0 and 1. Therefore, not more than $d-1$ branch points are allowed to vary and $\psi(\tilde{Z})$ must have dimension $\leq d-1$, as wanted.
\end{proof} 

\section{Chebychev polynomials and powers of solutions of the Pell-Abel equation} \label{Redei}

This section is devoted to describing the group of solutions of a Pell-Abel equation. As seen in the introduction, if a Pell-Abel equation $A^2-DB^2=1$ has a non-trivial solution, then it has infinitely many ones, obtained by taking powers $A_m+\sqrt D B_m=(A+\sqrt D B)^m$. Our goal for this section is to prove that $A_m^2=f_m(A^2)$, for some polynomial $f_m$ related to the $m$-th Chebychev polynomial. \\

We recall that Chebychev polynomials of the first kind are defined recursively as
\begin{equation*}
\begin{cases} T_0(t)=1, \\
T_1(t)=t, \\
T_{n+1}(t)=2tT_n(t)-T_{n-1}(t), \end{cases}
\end{equation*} 
so $T_n$ is a polynomial of degree $n$ and, for every $n>1$, the two terms of highest degree of $T_n$ are $2^{n-1}t^n$ and $-n2^{n-3}t^{n-2}$. Explicitly, we have that 
\begin{equation} \label{Chebychev}
T_{n}(t):=\sum _{h=0}^{[n/2]}\binom{n}{2h}t^{n-2h}(t^{2}-1)^{h},
\end{equation}
where $[\cdot]$ denotes the floor function. For a survey on these topics, see \cite{LMT}.

Suppose now that $(A,B)$ is a solution of the Pell- equation $A^2-DB^2=1$. Then, $(A_m,B_m)$ defined by $A_m+\sqrt D B_m=(A+\sqrt D B)^m$ is another solution of the same Pell-Abel equation, and
$A_m=\sum_{j=0}^{[m/2]} \binom{m}{2j} (DB^2)^j A^{m-2j}$. Using that $DB^2=A^2-1$, we have
\begin{equation}\label{Powerofsol}
A_m=\sum_{j=0}^{[m/2]} \binom{m}{2j} (A^2-1)^j A^{m-2j}=T_m(A).
\end{equation}
For the rest of the paper, we will denote by $\phi$ the square function $\phi(t)=t^2$, and by $\phi_g$ the composition $\phi \circ g=g^2$.
\medskip

For our purposes, it will be useful to express $\phi_{A_m}$ as a function of $\phi_A$. 
By \eqref{Powerofsol}, one can see that $\phi_{A_m}=f_m(\phi_A)$, where
\begin{equation} \label{f_m_even}
   f_{2k}(w):=\left ( \sum_{j=0}^{k} \binom{2k}{2j} (w-1)^j w^{k-j} \right )^2;
\end{equation}
and
\begin{equation} \label{f_m_odd}
   f_{2k+1}(w):=w\left ( \sum_{j=0}^{k} \binom{2k+1}{2j} (w-1)^j w^{k-j} \right )^2.
\end{equation}

For every $m\ge 1$ we call $f_m$ the \textit{$m$-th power polynomial}. Notice that $f_m \circ \phi=\phi \circ T_m$. 

For our purposes we will also need a description of the branch locus of the $f_m$; this can be easily done by looking at the branch locus of the Chebychev polynomials which is well known in the literature. We recall that, given a polynomial $f$ we say that a branch point $b$ of $f$ is of type $(\mu_1, \ldots, \mu_k)$ if the $\mu_i$ are the ramification indexes of the points in the preimage of $b$. This is also equal to the array of the multiplicities of the roots of $f(t)-b$.

\begin{proposition} \label{prop:ram_f_m}
The polynomials $f_m$ are branched only at $0$, $1$ and $\infty$. Moreover,
\begin{itemize}
\item if $m$ is even, then $0$ is of type $(2,2, \ldots, 2)$ and $1$ is of type $(1,1,2, \ldots,2)$; furthermore the two unramified points above $1$ are $0$ and $1$;
\item if $m$ is odd, then $0$ is of type $(1,2, \ldots, 2)$, and $0$ is the only unramified point above $0$; furthermore, $1$ is of type $(1,2, \ldots, 2)$, and $1$ is the only unramified point above $1$.
\end{itemize}
\end{proposition}

\begin{proof}
For $m=2$ we have that $f_m(t)=(2t-1)^2$, and the proposition holds trivially. We will then assume that $m\ge 3$.

The Chebychev polynomials are strictly related to Dickson polynomials $D_m(x,a)$ (for a definition of these polynomials and their properties see \cite[Section 3]{Bilu99}). Indeed, we have that $T_m(t)=\frac{1}{2}D_m(2t,1)$. Using this relation and \cite[Proposition 3.3, (b)]{Bilu99} we have that, if $m\ge 3$, the finite branching of $T_m$ happens only in $\pm 1$. If $m$ is odd, both $\pm 1$ are of type $(1,2, \ldots, 2)$, while if $m$ is even we have that $1$ is of type $(1,1,2, \ldots, 2)$ and $-1$ is of type $(2,\ldots,2)$.
Using that $\phi \circ T_m= f_m \circ \phi$, the finite branch points of $f_m$ are exactly the ones of $\phi \circ T_m$, i.e. $0$ and $1$. Moreover, 
 since $0$ is not a branch point for $T_m$, then the ramification points of $f_m\circ \phi$ have ramification index $2$.


Assume first that $m$ is even. By \eqref{f_m_even}, since $f_m$ is a square, then $0$ is a branch point for $f_m$. As $0$ is not contained $f_m^{-1}(0)$, then it is of type $(2,2, \ldots, 2)$. Let us now consider the branch point $1$; notice that $\{0,1\}\subseteq f_m^{-1}(1)$. Since $\phi$ is ramified in $0$, we have that $f_m$ is unramified in $0$. Using that $f_m \circ (1-t)=f_m$, this implies that $f_m$ is unramified also in $1$. By comparing the other ramification indexes, it follows that $1$ is of type $(1,1,2, \ldots, 2)$ as wanted.

Assume now $m$ odd. By \eqref{f_m_odd}, we have $f_m$ is unramified in $0$, and all the other points in the preimage of $0$ have ramification index $2$, so $0$ is of type $(1,2, \ldots, 2)$. Using the relation 
\begin{equation}
(1-t) \circ f_m \circ (1-t)= f_m, 
\end{equation}
we have that the type of $0$ and the type of $1$ are the same, and $1$ is the unramified point over $1$, concluding the proof.
\end{proof}
%

\section{Monodromy groups and polynomial decompositions}
\label{polynomial_decomposition_Sec}

The problem of finding (functional) decompositions of polynomials is very classical and has been studied first by Ritt in \cite{Ritt} and then by several authors (see, for instance, \cite{Fried73}). In this section we recall some basic facts that will be used in the paper.

Let $f$ be a non-constant polynomial in $\C[t]$, let $u$ be transcendental over $\C$ and $L$ be the splitting field of $f(t)-u$ over $\C(u)$. The \textit{monodromy group} $\Mon(f)$ of $f$ is the Galois group of $L$ over $\C(u)$, viewed as a group of permutations of the roots of $f(t)-u$. By Gauss’ lemma, it follows that $f(t) - u$ is irreducible over $\C(u)$, so $\Mon(f)$ is a transitive permutation group.
If $x$ is a root of $f(t)-u$ in $L$, then $u=f(x)$ and $\Mon(f)=\Gal(L/\C(f(x)))$.
We denote moreover by $H$ the stabilizer of $x$, i.e., $H=\Gal(L/\C(x))$. 
Note that, by Theorem \ref{thm_miranda}, the polynomial $f$ of degree $n$ seen as a map $\PP^1(\C)\rightarrow \PP^1(\C)$ corresponds to a conjugacy class of $h$-tuples of permutations  in $S_n$, where $h$ is the cardinality of the branch locus of $f$. Then, once we fix a labelling of the roots of $f(t)-u$ (which corresponds to fixing a representative $(\sigma_1, \ldots, \sigma_h)$ in the conjugacy class), the monodromy group of $f$ is isomorphic to the subgroup of $S_n$ generated by $\sigma_1, \ldots, \sigma_h$.

Applying L\"uroth's Theorem \cite[Theorem 2]{Schinzel} and \cite[Theorem 4]{Schinzel}, we have the following proposition.

\begin{proposition} \label{prop:correspondence_extensions_decompositions}
Given a polynomial $f\in \C[x]$, there is a correspondence between polynomial decompositions of $f$ and subfields of $\C(x)$ containing $\C(f(x))$. Moreover, if $\C(f(x)) \subseteq K \subseteq \C(x)$, then $f=g \circ h$ with $g,h$ polynomials and $\deg h=[\C(x): K]$. 
\end{proposition}

Using the Galois Correspondence, this shows that the study of the polynomial decompositions of $f$ reduces to the study of subgroups of the monodromy group of $f$ containing $H$. 


\noindent We give the following definition:
\begin{defn} \label{linear_equivalence}
We say that two polynomials $f,g\in \C[t]$ are \emph{linearly equivalent} if there exists two linear polynomials $\ell_1, \ell_2\in \C[t]$ such that $f= \ell_1 \circ g \circ \ell_2$.
\end{defn}

Notice that, if $f=g\circ h$, we can always change the decomposition up to composing with linear polynomials, so it makes sense to study the polynomial decompositions up to linear equivalence. Moreover, notice that two linearly equivalent polynomials have the same monodromy group.

We are now interested in computing the monodromy group of the polyomials $f_m$ defined in the previous section. 

\begin{proposition} \label{dihedral}
For every $m\ge 2$, the monodromy group of $f_m$ is exactly $D_{2m}$, where $D_{2m}$ denotes the dihedral group of order $2m$.
\end{proposition}

\begin{proof}
By Proposition \ref{prop:ram_f_m}, the polynomials $f_m$ are branched only in $0$ and $1$, both of type $(1,2 \ldots, 2)$ if $m$ is odd or of type $(2,\ldots, 2)$ and $(1,1,2\ldots, 2)$ respectively in $m$ is even. Applying \cite[Theorem 3.4]{Bilu99}, this implies that $f_m$ is linearly equivalent to the Dickson polynomial $D_m(x,a)$ with $a \neq 0$. By \cite[Theorem 3.6]{Bilu99}, the monodromy group of $D_m(x,a)$ with $a \neq 0$ is the dihedral group $D_{2m}$, which gives that $\Mon (f_m)=D_{2m}$ as wanted.
\end{proof}

\section{Characterization of primitivity}

In this section, we want to give a criterion to detect the primitivity of a solution of the Pell-Abel equation in terms of the associated conjugacy class of permutations. 
We start with a definition.

\begin{defn}\label{def:prim} For positive integers $n,\ell$ with $\ell\mid 2 n$, we call a partition $\mathcal{F}=\{F_1,\ldots, F_{\ell}\}$ of $\{1,\dots , 2n \}$ an \textit{$\ell$-partition} if $\mathcal{F}$ consists of $\ell$ subsets each of cardinality $\frac{2n}{\ell}$.
	A subgroup $G < S_{2n}$ is said to be \textit{$\ell$-imprimitive} if there exists an $\ell$-partition $\mathcal{F}=\{F_1,\ldots, F_{\ell}\}$ which is preserved by the action of $G$, i.e. for every $\sigma\in G$ and every $a\in \{1, \ldots, \ell\}$, there exists $b\in \{1, \ldots, \ell\}$ such that $\sigma(F_a)=F_b$. 
\end{defn}

It is well-known (see \cite{Ritt}) that, if $f= g \circ h$ is a polynomial of degree $r$ which is the composition of two polynomials $g$ and $h$, where $\deg h=s$ and $\deg g=r$, then $\Mon(f)$ is $r$-imprimitive. Indeed, if $q$ is not a branch point for $f$, then the set of its preimages via $f$ is partitioned in subsets whose elements map via $h$ to the same preimage of $q$ via $g$. 

We now give a criterion for determining whether a solution $(A,B)$ of a Pell-Abel equation is primitive or an $m$-th power for some $m\geq 2$, and this behavior is completely determined by the monodromy group $G_A$ of $\phi_A$. 
Note that $G_A$ is always $2$-imprimitive, since $\phi_A$ is the square of a polynomial.




\begin{theorem} \label{characterization}
	Let $(A,B)$ be a solution of degree $n$ of a Pell-Abel equation $A^2-DB^2=1$ with $\deg D=2d$ and let $G_A$ be the monodromy group of $\phi_A$. Then, for every integer $m\mid n$ with $\frac{n}{m} \ge d$, $(A,B)$ is the $m$-th power of another solution if and only if $G_A$ is $2m$-imprimitive. 
\end{theorem}

We will prove the above theorem in the next section. Before this, we show how Corollary \ref{Thm.prim} can be deduced from Theorem \ref{characterization}. 

%
%
%

\begin{proof}[Proof of Corollary \ref{Thm.prim}]
	We prove that a solution $(A,B)$ associated to the conjugacy class of the tuple $(\sigma_0,\sigma_{\infty},\sigma_1,\tau_1, \ldots,\tau_{d-1})$ given in Example  \ref{Ex.Zannier} is primitive for every $n\ge d \ge 2$.
	
	 We recall that $\sigma_{\infty}=(2n, \ldots, 1)$, $\sigma_0=(1,2n)$ $(2,2n-1)\cdots (n,n+1)$, $\sigma_1=(1,2n-1)(2,2n-2)\cdots (n-d,n+d)$ and $\tau_i=(n-i,n+i)$, for $i=1, \ldots, d-1$.
	 
Suppose by contradiction that $(A,B)$ is the $m$-th power of another solution $(A',B')$ for some $m \mid n$, $1<m<n$. 
By Theorem \ref{characterization}, $G_A$ has to preserve a $2m$-partition of $\{1, \ldots, 2n\}$. We see that the only $2m$-partition preserved by $\sigma_{\infty}$ is $\mathcal F_{2m}:=\{F_1,\ldots, F_{2m}\}$, where for every $1\le h\le 2m$, we set
$$F_h= \{ a \in \{1,\dots , 2n \}: a \equiv h \mod 2m \}.$$ 
	On the other hand, $\tau_{d-1}=(n-1, n+1)$ and $n-1,\ n+1$ are not in the same congruence class modulo $2m$ if $m\ge 2$, so $\mathcal F_{2m}$ is not preserved by $G_A$. This implies that $G_A$ does not preserve any $2m$-partition, hence proving the theorem.
\end{proof}

\section{Proof of Theorem \ref{characterization}}
 
In view of what we explained in the previous section, one direction is quite easy. Indeed, let $(A,B)$ be a solution of a Pell-Abel equation that is the $m$-th power of another solution $(A', B')$. In Section \ref{Redei} we showed that $\phi_{A}=f_m(\phi_{A'})=f_m \circ \phi \circ A'$, where $\phi(t)=t^2$. As $f_m$ is a polynomial of degree $m$, it follows immediately that $G_{A}$ is $2m$-imprimitive. \\

The proof of the converse is much more involved.

Let us consider a solution $(A,B)$ with $A$ of degree $n$ and the associated almost-Belyi map $\phi_A$. By Theorem \ref{thm_miranda}, the map $\phi_A$ corresponds to the conjugacy class of a $(k+3)$-tuple of permutations in $S_{2n}$ given by $\Sigma_A:=(\sigma_0,\sigma_{\infty},\sigma_1,\tau_1, \ldots,\tau_{k})$ with $k \le d-1$, where we recall that $\sigma_{\infty}$ is a $2n$-cycle, $\sigma_0$ is the product of disjoint cycles of even length, $\sigma_1$ fixes exactly $2d$ indexes and is the product of disjoint cycles of even length. We call $b_1, \ldots, b_k$ the further $k$ branch points different from $0,1, \infty$. 

Conjugating by a suitable permutation, we can and will always assume that $\sigma_{\infty}=(2n,2n-1,\ldots, 1)$. 
We recall that the length of the disjoint cycles appearing in the permutations of $\Sigma_A$ is very restricted. In particular, we know that $\sigma_1$ fixes $2d$ indexes and $\sum_{i,j}(m_{ij}-1)\leq d-1$ where $(m_{i1},\dots, m_{ih_i})$ is the cycle structure of $\tau_i$. Therefore, there are at least two indexes which are fixed by $\sigma_1$ and, at the same time, by every $\tau_i$. Without loss of generality, we  are going to assume that one of them is $2n$. 
Moreover, each $\tau_i$ must fix at least $2n-2(d-1)$ indexes.

%

The monodromy group $G_A$ of $\phi_A$ is the subgroup of $S_{2n}$ generated by $\sigma_0,\sigma_{\infty},\sigma_1,\tau_1, \ldots,\tau_{k}$; by assumption, $G_A$ preserves a $2m$-partition of $\{1, \ldots ,2n\}$. As mentioned before, the only $2m$-partition preserved by $\sigma_{\infty}$ is $\mathcal F_{2m}=\{F_1,\ldots, F_{2m}\}$, where, for every $h=1,\ldots, 2m$, we defined $F_h=\{j\in \{1, \ldots, 2n\} : j\equiv h \mod 2m\}$; we can then assume that $G_A$ preserves this specific partition.


For $h=1, \ldots, m$ we let $E_h=\{ j\in \{1, \ldots, 2n\} : j\equiv h \mod m \}$. Then, $\mathcal E_m=\{E_1, \ldots, E_m\}$ gives another partition of $\{1,\ldots, 2n\}$ and we have that $E_h=F_h \cup F_{m+h}$. We start by proving the following lemma.

\begin{lemma}\label{lemma:structure_G}
Suppose $G_A$ preserves the $2m$-partition $\mathcal F_{2m}$; then, $G_A$ preserves the $m$-partition $\mathcal E_m$. More precisely, the action of $\sigma_0,\sigma_{\infty},\sigma_1,\tau_1, \ldots,\tau_{k}$ on $\mathcal F_{2m}$ and $\mathcal E_m$ is the following:
\begin{equation} 
\begin{cases}
\sigma_{\infty}(F_1)=F_{2m} \mbox{ and }  \sigma_{\infty}(F_i)= F_{i-1} \ \forall{i=2,\ldots, 2m};\\
\sigma_1(F_{2m})=F_{2m} \mbox{ and }  \sigma_1(F_i)=F_{2m-i}\ \forall i=1, \ldots, 2m-1; \\
\sigma_0(F_i)=F_{2m-i+1} \quad \forall i=1,\ldots, 2m; \\
\tau_j(F_i)=F_i \ \forall i=1, \ldots, 2m \mbox{ and } j=1,\ldots, k.
\end{cases} 
\end{equation}
and:
\begin{equation} 
\begin{cases}
\sigma_{\infty}(E_1)=E_{m} \mbox{ and }  \sigma_{\infty}(E_i)= E_{i-1} \ \forall{i=2,\ldots, m};\\
\sigma_1(E_m)=E_{m} \mbox{ and }  \sigma_1(E_i)=E_{m-i}\ \forall i=1, \ldots, m-1; \\
\sigma_0(E_i)=E_{m-i+1}\quad \forall i=1,\ldots, m; \\
\tau_j(E_i)=E_i \ \forall i=1, \ldots, m \mbox{ and } j=1,\ldots, k.
\end{cases} 
\end{equation}
In particular, $\mathrm{Fix}(\sigma_1):=\{a\in \{1,\ldots, 2n\}\ |\ \sigma_1(a)=a\}\subseteq E_m$. 
\end{lemma}

\begin{proof}
By construction $\sigma_1(2n)=2n$, thus we have $\sigma_1(F_{2m})=F_{2m}$. Moreover, we have that $|F_h|=n/m \ge d$ for all $h$, and for all $i$ at most $2d-2$ indexes are not fixed by $\tau_i$. This implies that $\tau_i(F_h)=F_h$ for every $i=1, \ldots, k$ and $h=1, \ldots, 2m$.

In what follows the indexes of the $F_h$ are considered modulo $2m$.
Using $\sigma_0\sigma_{\infty}\sigma_1\prod_{i=1}^{k}\tau_i=\mathrm{id}$, we deduce that $\sigma_0\sigma_{\infty}\sigma_1(F_h)=F_h$ for every $h=1, \ldots, 2m$. As we know that $\sigma_{\infty}(
F_{h})=F_{h-1}$, this gives that
$$\begin{cases} 
\sigma_0(F_{i_1})=F_{i_2} \ \Rightarrow\ \sigma_1(F_{i_1})=F_{{i_2}-1}; \\ 
\sigma_1(F_{j_1})=F_{j_2} \ \Rightarrow\  \sigma_0(F_{j_2})=F_{i_1+1}.
\end{cases} $$
Since $\sigma_1(F_{2m})=F_{2m}$, we have that
$$ \sigma_1(F_{2m})=F_{2m} \quad \mbox{ and } \quad \sigma_1(F_i)=F_{2m-i}\ \forall i=1, \ldots, 2m-1, $$
and $$ \sigma_0(F_i)=F_{2m-i+1} \quad \forall i=1,\ldots, 2m. $$ 
Now, for every $h=1,\ldots, m$, the set $E_h$ is equal to $F_h \cup F_{m+h}$. Using the previous relations, we have that $\sigma_1(E_h)=F_{2m-h}\cup F_{m-h}=E_{m-h}$; $\sigma_0(E_h)=F_{2m-h+1}\cup F_{m-h+1}=E_{m-h+1}$, $\tau_i$ stabilizes $\{E_1, \ldots, E_m\}$ for evey $i=1, \ldots, d-1$ and $\sigma_{\infty}(E_h)=F_{h-1}\cup F_{m+h-1}=E_{h-1}$, which concludes the proof.
\end{proof}


For $G < S_{2n}$ and $C_1,\ldots, C_{\ell}\subseteq \{1, \ldots, 2n\}$, we define
$$ \Stab_G(C_1, \ldots, C_{\ell}):=\bigcap_{i=1}^{\ell} \Stab_G(C_i), $$
where, for every $i=1, \ldots, \ell$,
$$ \Stab_G(C_i)=\{\sigma\in G : \sigma(C_i)=C_i \}. $$
%
First, we prove this preliminary lemma.
\begin{lemma} \label{lemma:subgroups}
Let us assume that $G_A$ preserves $\mathcal F_{2m}$ and set $S:=\Stab_{G_A}(\{2n\})$. Then, there exist $H_1, H_2$ subgroups of $G_A$ satisfying $S \subseteq H_1< H_2 < G_A$ with $[G_A:H_2]=m$ and $[H_2:H_1]=2$. Moreover, if $m>2$, then there exists a normal subgroup $K$ of $G_A$ containing $S$ and contained in $H_2$ such that $G_A/K \cong D_{2m}$.
\end{lemma}

\begin{proof}
As by assumption $G_A$ preserves $\mathcal{F}_{2m}$, by Lemma \ref{lemma:structure_G} it preserves also $\mathcal{E}_m$, hence it induces an action on $\{F_1, \ldots, F_{2m}\}$ and on $\{E_1, \ldots, E_m\}$. 

Let us define $H_1:=\Stab_{G_A}(F_{2m})$ and $H_2:=\Stab_{G_A}(E_m)$.  As $2n \in F_{2m} \subseteq E_m$, by Lemma \ref{lemma:structure_G} it follows that $S\subseteq H_1 \subseteq H_2$. 

The group $G_A$ acts transitively on $\{1,\ldots, 2n\}$; therefore, the action on $\{F_1, \ldots, F_{2m}\}$ and on $\{E_1, \ldots, E_m\}$ is transitive as well, so the orbits of $F_i$ and $E_j$ have cardinality $2m$ and $m$ respectively. This implies that $|G_A|/|H_2|=|G_A\cdot E_m|=m$ and $|G_A|/|H_1|=|G_A\cdot F_{2m}|=2m$, where $G_A \cdot C$ denotes the orbit of $C$ with respect to the action of $G_A$. This proves that $[G_A:H_2]=m$ and $[H_2:H_1]=2$ as wanted.

We now prove the second part of the statement. 

For $m>2$ let us define $K:= \Stab_{G_A}(E_1, \ldots, E_m)$. Note that $K \subseteq H_2$ and, if $m >2$, then $K\neq H_2$ (as $\sigma_1 \in H_2$ and $\sigma_1 \not \in K$). To conclude the proof, we have to show that $K$ is a normal subgroup of $G_A$ and $G_A/K\cong D_{2m}$.

First, let us prove that $K\lhd G_A$. Take $h\in K$ and $g\in G_A$; then, if $g(E_i)=E_{j_i}$ for some $j_i$, we have that  
\[ 
g^{-1}hg(E_i)=g^{-1}h(E_{j_i})=g^{-1}(E_{j_i})=E_i \quad \mbox{for all } i=1,\ldots, m, 
\]
so $g^{-1}hg \in K$ as wanted. 



Finally, we show that $G_A/K\cong D_{2m}$. 
As $G_A$ induces an action on $\{E_1,\ldots, E_m\}$, we can define a homomorphism $\varphi: G_A \rightarrow S_m$ given by $\varphi(\alpha)=\beta$ where $\beta \in S_m$ is defined by $\beta(i)=j$ if $\alpha(E_i)=E_j$. Since $K=\Stab_{G_A}(E_1, \ldots, E_m)$, then $\varphi$ induces an injective homomorphism $\tilde \varphi: G_A/K \rightarrow S_m$. We want now to prove that $\tilde \varphi(G_A/K)\cong D_{2m}$.
By Lemma \ref{lemma:structure_G}, $\tilde \varphi(\tau_i)=\mathrm{id}$ for every $i=1, \ldots, k$, so $\tilde \varphi(G_A/K)$ is generated by the images of $\sigma_{\infty}$ and $\sigma_1$, which we denote by $r$ and $s$ respectively. We have to prove that $r^m=\mathrm{id}$, $s^2=\mathrm{id}$ and $srsr=\mathrm{id}$. Notice that, if $\sigma_{\infty}=(2n, 2n-1, \ldots, 1)$, then $r=(m,m-1, \ldots, 1)$, so $r^m=\mathrm{id}$. Moreover, by Lemma \ref{lemma:structure_G}, $s$ is a product of traspositions, so it has order $2$ as wanted. 
Let us finally prove that $srsr=\mathrm{id}$. By Lemma \ref{lemma:structure_G} we have that $\sigma_1(E_m)=E_m$, $\sigma_1(E_i)=E_{m-i}$ for all $i=1,\dots, m-1$ and $\sigma_{\infty}(E_1)=E_{m}$ and $\sigma_{\infty}(E_j)=E_{j-1}$ for all $j=2, \ldots, m$. So, for every $i=1,\ldots, m$, 
\[
srsr(E_i)=rsr(E_{m-i})=sr(E_{m-i-1})=r(E_{i+1})=E_i.
\]
This shows that $G_A/K\cong D_{2m}$, concluding the proof.
\end{proof}

Using the previous lemma, we can finally prove the following result about the polynomial decomposition of $\phi_A$, which concludes the proof of Theorem \ref{characterization}.

\begin{proposition} \label{suff_prop}
Suppose $G_A$ preserves $\mathcal{F}_{2m}$ with $m\ge 2$ and $n/m \ge d$. Then, $\phi_A=f_m(\phi_{A'})$, where $A'$ is another solution of the same Pell-Abel equation.
\end{proposition}

\begin{proof}
By Lemma \ref{lemma:subgroups}, there exists a chain of groups 
$$ \Stab_{G_A}(\{2n\}) < H_1 < H_2 < G_A,  $$
where $[G_A:H_2]=m$ and $[H_2:H_1]=2$, where we recall that $H_1=\Stab_{G_A}(F_{2m})$ and $H_2=\Stab_{G_A}(E_m)$. By Galois theory, this implies that there exists a tower of subfields of $\C(t)$ 
$$ T:=\C(\phi_A) \subset L_2 \subset L_1 \subset \C(t), $$
with $[L_2:\C(\phi_A)]=m$ and $[L_1:L_2]=2$.
Together with Proposition \ref{prop:correspondence_extensions_decompositions}, this implies that 
$$ \phi_A=h_2 \circ h_1 \circ z, $$
where $h_1,h_2$ and $z$ are polynomials with $\deg h_1=2$ and $\deg h_2=m$. 

Let us first prove that $h_2$ is linearly equivalent to $f_m$. If $m=2$, this is trivial since polynomials of degree $2$ are linearly equivalent to each other; so, let us assume that $m>2$. 
By Lemma \ref{lemma:subgroups}, the subgroup $K=\Stab_{G_A}(E_1, \ldots, E_m)$ is a normal subgroup of $G_A$ contained in $H_2$ and such that $G_A/K\cong D_{2m}$. Then, $K$ corresponds to a field $F\subseteq \C(t)$ containing $L_2$ and such that $F/T$ is a normal extension.
We want to show that $F/T$ is the Galois closure of $L_2/T$.

As the action of $G_A$ on $\{E_1, \ldots, E_m\}$ is transitive, the subgroups $\Stab_{G_A}(E_1), \ldots, \Stab_{G_A}(E_m)$ are exactly all the conjugates of $\Stab_{G_A}(E_m)=\Gal(\C(t)/L_2)$ in $G_A$. Therefore the Galois closure of $L_2/T$ will be equal to the subfield of $\C(t)$ corresponding to $\Stab_{G_A}(E_1)\cap \ldots \cap \Stab_{G_A}(E_m)=K$ that is $F$ as wanted. 

As by construction $\Gal(F/K)=\Mon(h_2)\cong D_{2m}$, by Proposition \ref{dihedral} we have that the polynomial $h_2$ is linearly equivalent to the Chebychev polynomial $T_m$, and so to $f_m$ as proved in Section \ref{polynomial_decomposition_Sec}. This gives that there exist linear polynomials $\alpha_1$ and $\alpha_2$ such that
\[ \phi_A=\alpha_1\circ f_m \circ \alpha_2 \circ h_1 \circ z= \alpha_1\circ f_m \circ h_3 \circ z, \]
where $h_3=\alpha_2 \circ h_1$ is a polynomial of degree $2$. 

We want to prove that $\phi_A= f_m \circ \phi \circ A'$ where $ A'$ fits in the same Pell-Abel equation of $A$.

For the rest of the proof, as we are dealing with polynomials, we only consider branching at finite points.

 By Section \ref{Redei}, $f_m$ has two branch points, namely $0$ and $1$. As $\alpha_1$ is a linear map, $\alpha_1\circ f_m$ will be branched only at $\xi_0=\alpha_1(0)$ and $\xi_1=\alpha_1(1)$ with the same ramification indexes. As the branch points of $\alpha_1\circ f_m$ lie among the ones of $\phi_A$, this means that $\xi_0, \xi_1 \in \{0,1,b_1, \ldots, b_{k}\}$. We now consider two cases and write $\alpha_1(t)=pt+q$. We recall that in Proposition \ref{prop:ram_f_m} we gave a characterization of the ramification behavior of $f_m$ and we are going to use this for the rest of the proof.

 If $m$ is even, as 0 is the only point with no simple preimage via $\phi_{A}$, we necessarily have $\xi_0=0$ and thus $q=0$. If $m>2$ we have that the ramification above $\xi_1$ counts at least $(m-2)/2 \cdot (n/m)\geq (m/2-1)d>d-1$ while above any of $b_1,\dots, b_k$ counts at most $d-1$. This forces $\xi_1=1$ and $p=1$. If $m=2$ we have $\alpha_1\circ f_2=p (2t-1)^2=(2(\sqrt{p} t +(1-\sqrt{p})/2)-1)^2=f_2\circ (\sqrt{p} t +(1-\sqrt{p})/2)$. In any case we have or we may suppose that $\alpha_1=t$.

Assume now $m$ is odd; then, the ramification above $\xi_0$ counts at least $(m-1)/2 \cdot (n/m)\geq d>d-1$ and the same for $\xi_1$. As before, above $b_1,\dots, b_k$ it counts at most $d-1$. This implies that $\{\xi_0,\xi_1\} = \{0,1\}$. In particular, if $\xi_0=0$ and $\xi_1=1$, then $\alpha_1=t$ while if $\xi_0=1$ and $\xi_1=0$ we have $\alpha_1=1-t$. We can reduce to the first case by noticing that $(1-t) \circ f_m = f_m \circ (1-t)$.
 
%
%
%
%

Finally, we have
\begin{equation*} 
 \phi_A= f_m \circ h_3 \circ z, 
\end{equation*}
where $h_3$ is a polynomial of degree $2$.

We write $w=h_3 \circ z$ and let $r$ be the unique finite branch point of $h_3$. We claim that $r\in \{0,1\}$. Suppose not; the fact that $f_m(\{0,1\})\subseteq \{0,1\}$ and that $\phi_A$ fits in a Pell-Abel equation implies that $w^{-1}(\{0,1\}) $ contains exactly $2d$ simple preimages and thus the sum of the ramification above 0 and 1 of $w$ counts at least $2(n/m)-d$. On the other hand, the ramification of $w$ above $r$ counts at least $n/m$ making the whole ramification to count at least $3(n/m)-d\geq 2(n/m)=\deg w$. This contradicts the Riemann-Hurwitz formula and we have $r\in \{0,1\}$.

If $m$ is odd, then we must have $r=0$ because $1$, being the only simple preimage of 1 via $f_m$ must have simple preimages via $w$.

If $m$ is even and $r=1$ we observe that $f_m=f_m\circ (1-t) $. We then can simply replace $h_3$ by $(1-t)\circ h_3$ and thus we may assume $r=0$.

In any case we have $h_3=st^2=(\sqrt{s}t)^2$ and then, possibly composing with a suitable linear polynomial on the right, we can assume that $h_3$ is exactly the square function $\phi$, so 
$$ \phi_A= f_m \circ \phi \circ A'. $$ 
We are left with proving that $A'$ is another solution of the same Pell-Abel equation, i.e. that $\phi_{A'}-1=D B'^2$ for some polynomial $B'$. 
Notice that, since $A$ is a solution of the Pell-Abel equation, $\phi_A^{-1}(1)$ contains exactly $2d$ simple points $\{r_1, \ldots, r_{2d}\}$ (the zeros of $D$); as $1$ (and $0$ if $m$ is even) is the only point in $f_m^{-1}(1)$ with ramification index $1$ and $\phi_{A'}=\phi \circ A'$ is a square, then $\{r_1, \ldots r_{2d}\} \subset \phi_{A'}^{-1}(1)$. This shows that $(\phi_{A'}-1)/D={B'}^2$ for some polynomial $B'$ as wanted.
\end{proof}

\begin{corollary}  \label{cor:behav_perm}
Let $m$ be a positive integer.
	Let $(A,B)$ be a solution of degree $n$ of the Pell-Abel equation $A^2-DB^2=1$ and, for $i=1, \ldots, 2m$, let $F_i=\{j\in \{1, \ldots, 2n\}\ |\ j\equiv i \mod 2m\}$ and, after conjugating, assume that $\sigma_{\infty}=(2n, 2n-1, \ldots, 1)$ and that one of the  points which are fixed by $\sigma_1$ and by all the $\tau_i$ is $2n$. Then, $(A,B)$ is the $m^{th}$-power of a solution $(A', B')$ of the same Pell-Abel equation if and only if the permutations $\sigma_{\infty}, \sigma_0, \sigma_1, \tau_1, \ldots, \tau_{k}$ associated to $A$ satisfy the following conditions:
	\begin{equation} \label{behaviour_perm}
	\begin{cases}
	\sigma_{\infty}(F_1)=F_{2m} \mbox{ and }  \sigma_{\infty}(F_h)= F_{h-1} \quad \forall{h=2,\ldots, 2m};\\
	\sigma_1(F_{2m})=F_{2m} \mbox{ and }  \sigma_1(F_h)=F_{2m-h}\quad \forall h=1, \ldots, 2m-1; \\
	\sigma_0(F_i)=F_{2m-i+1} \quad \forall i=1,\ldots, 2m; \\
	\tau_j(F_i)=F_i \quad \forall i=1, \ldots, 2m \mbox{ and } j=1,\ldots, k.
	\end{cases} 
	\end{equation}
\end{corollary}

\begin{ese}
Let us see an example with $n=6$ and $d=2$, so the permutations are in $S_{12}$. In this case, we have that the solution of the Pell-Abel equation can be either primitive or a square of a solution of degree 3 or the cube of a solution of degree 2. 

Consider $\sigma_{\infty}=(12, 11, \ldots, 1)$ and $\sigma_0=(1, 12)(2,11)(3,10)(4,9)(5,8) (6,7)$. As $\sigma_0\sigma_{\infty}\sigma_1\tau_1=1$, we have that $\sigma_1\tau_1=(1,11)(2,10)(3,9)(4,8)(5,7)$. So the solution will depend on the choice of $\tau_1$ among these transpositions. In particular, if we choose $\tau_1=(1,11)$ or $\tau_1=(5,7)$, the associated solution is primitive. In fact, by \eqref{behaviour_perm}, we have that the solution associated to these permutations is neither a square nor a cube (otherwise, $1, 11$ or $5,7$ would lie in the same congruence class modulo $4$ in the case of a square or modulo $6$ in the case of a cube). Moreover, if $\tau_1=(2,10)$ or $\tau_1=(4,8)$, the solutions are squares. In fact, in this case it is easy to see that $(\sigma_{\infty},\sigma_0,\sigma_1,\tau_1)$ satisfies \eqref{behaviour_perm} for $m=2$, while, if $\tau_1=(3,9)$, the solution is a cube. 
Note however that the two 4-tuples with $\tau_1=(2,10)$ and $\tau_1=(4,8)$ are actually conjugated by $\sigma_{\infty}^6$. 
\end{ese}

\section{Counting the conjugacy classes of permutations} \label{sec:d=2}

In this section, we want to show how to apply Theorem \ref{thm_miranda} and use combinatorial arguments to count equivalence classes of Pellian polynomials with a solution of degree $n$.

As discussed in the introduction, the advantage of the combinatorial argument is that it allows to compute the precise number instead of only an asymptotic formula, and this can be done algorithmically given $n$ and $d$. On the other hand, the combinatorics becomes more difficult as soon as $d$ grows, due to the number of different configurations that the permutations can have. For this reason, in this section we are going to stick to the case $d=2$ and to maximal branching, that is when the map $\phi_A$ is branched over $0$, $1$, $\infty$ and over another point which we will denote by $b$. In this case, using the arguments of Section \ref{ABm} and the Riemann-Hurwitz formula, if we count the branching of $\phi_A$ as the sum of $e-1$ over the ramification indices $e$, we have that, above $0$ the branching is exactly $n$, above $1$ is $n-1$, above $\infty$ it is $2n-1$ and above $b$ is $1$. This implies that the cycle decomposition of $\sigma_0, \sigma_{\infty}, \sigma_1$ and $\tau_1=:\tau$ is fixed, and we have that:
\begin{enumerate}
\item[1 -] $\sigma_0$ is the product of $n$ disjoint transpositions;
\item[2 -] $\sigma_{\infty}$ is a $2n$-cycle;
\item[3 -] $\sigma_1$ is the product of $n-1$ disjoint transpositions;
\item[4 -] $\tau$ is a transposition.
\end{enumerate}
Moreover, we know that $\sigma_0, \sigma_{\infty}, \sigma_1$ and $\tau$ satisfy $\sigma_0 \sigma_{\infty}\sigma_1\tau=$id.

As said before each conjugacy class of $4$-tuples $(\sigma_0,\sigma_{\infty}, \sigma_1, \tau)$ has at least one representative with $\sigma_{\infty}=(2n,2n-1, \ldots, 1)$ since $\sigma_{\infty}$ is a $2n$-cycle and one can relabel the element using a suitable conjugation. Moreover, there are at least two indexes that are fixed by $\sigma_1$ and by $\tau$; without loss of generality we will always assume that one of the two indexes is $2n$. We will call such a representative \textit{special}.

Now, since $\sigma_1$ is the product of $n-1$ disjoint transpositions and $\tau$ is a transposition, we have three different cases:
\begin{itemize}
\item $\sigma_1$ and $\tau$ are disjoint permutations;
\item $\sigma_1\tau$ is the product of $n-2$ transpositions and a $3$-cycle;
\item $\sigma_1\tau$ is the product of $n-3$ transpositions and a $4$-cycle.
\end{itemize}
In what follows, we are going to study the special $4$-tuples, analysing in particular how many special $4$-tuples we have in each conjugacy class. As the conjugation preserves the cycle decomposition, we
have to analyse these three cases separately.

\subsection{The disjoint case} \label{disjoint}

Let $\Sigma$ be a special $4$-tuple, i.e. $\sigma_{\infty}=(2n, 2n-1, \ldots, 1)$ and $\sigma_1$ and $\tau$ both fix $2n$. Suppose moreover that $\sigma_1$ and $\tau$ are disjoint permutations.

Since $2n$ is fixed by both $\sigma_1$ and $\tau$ and $\sigma_0\sigma_{\infty}\sigma_1\tau=$id, then $2n$ is fixed by $\sigma_0\sigma_{\infty}$; moreover, as $\sigma_0\sigma_{\infty}\sigma_1\tau=$id, we must have that $\sigma_0(2n)=1$. Since $\sigma_0$ is composed only by transpositions, we have also that $\sigma_0(1)=2n$, hence $\sigma_1\tau(2n-1)=1$. Now as $\sigma_1\tau$ is composed only by transpositions, we have that  $\sigma_1\tau(1)=2n-1$, so $\sigma_0(2n-1)=2$ and so on.
Going on with this argument, we have that $\sigma_0$ must have the form $(1,2n)(2,2n-1) \cdots (n,n+1)$ and $\sigma_0\sigma_{\infty}$ fixes also $n$. In this case, $\sigma_1 \tau=(1,2n-1)(2,2n-2) \cdots (n-1, n+1)$ and we may choose as $\tau$ any of the $n-1$ transpositions $(1,2n-1),(2,2n-2), \cdots, (n-1, n+1)$. 
\medskip

Now, assume that $\Sigma'=(\sigma_0',\sigma_{\infty},\sigma_1',\tau')$ is another special $4$-tuple lying in the same conjugacy class of $\Sigma$, i.e. $\Sigma'=\gamma^{-1}\Sigma \gamma$ for some $\gamma \in S_{2n}$; then, since by assumption $\sigma_{\infty}$ has to be fixed by the conjugation, we have that $\gamma$ will be a power of $\sigma_{\infty}$. Furthermore, since we want $\sigma_1'$ and $\tau'$ both  to fix $2n$, then the conjugation has to permute $2n$ with $n$ (which is the only other point fixed by $\sigma_1$ and $\tau$); consequently, $\gamma=\sigma_{\infty}^n$. Notice finally that, given a special $4$-tuple $\Sigma$, we have $\sigma_{\infty}^{-n}\Sigma \sigma_{\infty}^n= \Sigma$ if and only if $n$ is even and $\tau=\left ( \frac{n}{2}, 2n-\frac{n}{2} \right )$. 

This implies that:
\begin{itemize}
   \item if $n$ is odd, then we have $\frac{1}{2}(n-1)$ conjugacy classes, because every conjugacy class contains two special $4$-tuples;
   \item if $n$ is even, then we have $\frac{1}{2}n$ conjugacy classes, because every conjugacy class contains exactly two special $4$-tuples, except for the conjugacy class of the $4$-tuple with $\tau=\left (\frac{n}{2}, 2n-\frac{n}{2}\right )$ which contains only one special $4$-tuple.
\end{itemize}
In general we have
\begin{equation} \label{number_conj_classes}
\#\ \textrm{conjugacy classes}=
\left \lfloor \frac{n}{2} \right \rfloor.
\end{equation}

We point out that, in the disjoint case, Corollary \ref{cor:behav_perm} gives an easier way to detect whether the solution $(A,B)$ is primitive by looking only at the corresponding $\tau$ of a special representative.

\begin{proposition}
Consider the $4$-tuple $\Sigma=(\sigma_0,\sigma_{\infty},\sigma_1,\tau)$ with $\sigma_{\infty}=(2n,2n-1,\ldots,1)$ and $\sigma_0=(1,2n)(2,2n-1)\ldots (n,n+1)$ associated to a solution $(A,B)$ of a Pell-Abel equation. Then, $(A,B)$ is the $m$-th power of another solution if and only if $\tau=(h,2n-h)$ with $m\mid h$.

In particular, there are exactly $\varphi(n)/2$ equivalence classes corresponding to primitive solutions.
\end{proposition}
\begin{proof}
By Corollary \ref{cor:behav_perm}, $\Sigma$ corresponds to an $m$-th power if and only if \eqref{behaviour_perm} holds. The conditions on $\sigma_0$, $\sigma_\infty$ are always satisfied. The condition on $\tau$ is satisfied if and only if $\tau=(h,2n-h)$ with $h \equiv 2n-h\equiv -h \pmod {2m}$, i.e., $m\mid h$. Finally, if $m\mid h$, also $\sigma_1$ satisfies the conditions of \eqref{behaviour_perm}, proving the claim.
\end{proof}

\subsection{The 3-cycle case} \label{3-cycle}

Let $\Sigma=(\sigma_0,\sigma_{\infty},\sigma_1,\tau)$ be a special $4$-tuple, i.e. with $\sigma_{\infty}=(2n, \ldots, 1)$ and such that $\sigma_1$ and $\tau$
fix both $2n$, and assume that $\sigma_1\tau$ decomposes as a product of $n-3$ transpositions and a 3-cycle.

First, we are going to prove the following result that describes the possible shapes of $\sigma_0$ and $\sigma_1\tau$ of such a $4$-tuple.

\begin{proposition}
Let $\Sigma=(\sigma_0,\sigma_{\infty}, \sigma_1, \tau)$ be a special $4$-tuple such that $\sigma_1\tau$ contains a $3$-cycle. Then, there exist $h,k$ with $1\le h\leq n-2$, $h< k <2n-h$ and $h\equiv k \pmod 2$ such that 
\begin{align} \label{shape_3_cycle}
\sigma_0 &=\prod_{i=1}^h(i,2n+1-i) \prod_{j=1}^{\frac{k-h}{2}}(h+j, k+1-j) \prod_{t=1}^{\frac{2n-h-k}{2}}(k+t, 2n-h+1-t);\\
\sigma_1\tau & =\prod_{i=1}^{h-1}(i,2n-i)\prod_{j=1}^{\frac{k-h}{2}-1} (h+j, k-j) \prod_{t=1}^{\frac{2n-h-k}{2}}(k+t,2n-h-t)\ (2n-h, h, k). \nonumber
\end{align}
In particular, $\sigma_1\tau$ fixes 3 indexes: $2n, \frac{k+h}{2}$ and $\frac{2n-h+k}{2}.$
Moreover, there are three different $\Sigma$ satisfying \eqref{shape_3_cycle} corresponding to different choices of $\tau$, i.e. $\tau\in \{(h,k), (h, 2n-h), (k, 2n-h)\}$.
\end{proposition}

\begin{proof}
Let $\Sigma=(\sigma_0,\sigma_{\infty},\sigma_1,\tau)$ be a special $4$-tuple and assume that $\sigma_1\tau$ is a product of $n-3$ disjoint transpositions and a disjoint 3-cycle. In this case, $\sigma_1\tau$ fixes $3$ indexes, and one of them is $2n$. Then, since $\sigma_0\sigma_{\infty}\sigma_1\tau=\mathrm{id}$, we must have that $\sigma_0(2n)=1$. As $\sigma_0$ is the product of $n$ disjoint transpositions, we have also that $\sigma_0(1)=2n$, hence $\sigma_1\tau(2n-1)=1$. Now, we have two choices: either $\sigma_1\tau(1)=2n-1$ (so $(2n-1,1)$ is one of the disjoint transpositions in the decomposition of $\sigma_1\tau$), or $\sigma_1\tau(1)=k\neq 2n-1$ (giving rise to the 3-cycle in the product). We can go on as in the disjoint case described in the previous section until we reach the situation in which there exist $1\le h\le n-2$ and $h+1\le k\le 2n-h-1$ such that $\sigma_0$ contains the disjoint transpositions $(1, 2n),(2,2n-1)\cdots (h,2n-h+1)$ and $\sigma_1\tau$ contains the disjoint transpositions $(1,2n-1)\ldots, (2n-h+1,h-1)$ and the 3-cycle $(2n-h,h,k)$. 

Notice that $k\neq h+1, 2n-h-1$, otherwise $\sigma_0\sigma_{\infty}\sigma_1\tau=\id$ would imply that $\sigma_0$ fixes $h+1$ (or in the second case $2n-h-1$), which is impossible because $\sigma_0$ has no fixed points. So, $h+1 <k< 2n-h-1$. Now, if $\sigma_1\tau(h)=k$, as $\sigma_0\sigma_{\infty}\sigma_1\tau=\id$, we have $\sigma_0(k)=h+1$ and so $\sigma_0(h+1)=k$. This means that $\sigma_1\tau(h+1)=k-1$ and $\sigma_1\tau(k-1)=h+1$ and so on. This process ends when one reaches the conditions $\sigma_0\left ( \frac{k+h}{2}\right)=\frac{k+h}{2}+1$ and viceversa, $\sigma_1\tau\left (\frac{k+h}{2}-1\right )=\frac{k+h}{2}+1$ and viceversa and $\frac{k+h}{2}$ is fixed by $\sigma_1\tau$. Notice that this gives the additional condition that $h\equiv k \pmod 2$. 

On the other hand, if $\sigma_1\tau$ contains the 3-cycle $(2n-h,h,k)$, we also have that $\sigma_1\tau(k)=2n-h$, giving other conditions to satisfy. In fact,  $\sigma_0\sigma_{\infty}\sigma_1\tau=\id$ implies $\sigma_0(2n-h)=k+1$ and $\sigma_0(k+1)=2n-h$. This means, as before, that $\sigma_1\tau(2n-h-1)=k+1$ and viceversa, and so on. This process ends when one reaches the conditions $\sigma_0\left ( \frac{2n-h+k}{2}\right)=\frac{2n-h+k}{2}+1$ and viceversa, $\sigma_1 \tau\left ( \frac{2n-h+k}{2}-1\right)=\frac{2n-h+k}{2}+1$ and $\frac{2n-h+k}{2}$ is fixed by $\sigma_1\tau$, proving the first part of the assertion.

Finally, let us denote by 
\[
\alpha_{hk}:= \sigma_1\tau (2n-h, h, k)^{-1},
\]
i.e. the product of the $n-2$ disjoint transpositions that appear in the decomposition of $\sigma_1\tau$.

Notice that, if $\sigma_1\tau$ contains the $3$-cycle $(2n-h, h, k)$, then we have three choices for $\sigma_1$ and $\tau$, namely $\sigma_1=\alpha_{hk} (2n-h, h)$ and $\tau=(2n-h,k)$ or $\sigma_1=\alpha_{hk} (h,k)$ and $\tau=(h, 2n-h)$ or $\sigma_1=\alpha_{hk}(k, 2n-h)$ and $\tau=(k,h)$, giving three different special $4$-tuples as wanted.
\end{proof}

We now want to count how many special $4$-tuples are contained in each conjugacy class. First of all, notice that if we have two special $4$-tuples $\Sigma=(\sigma_0,\sigma_{\infty},\sigma_1,\tau)$ and $\Sigma'=(\sigma_0',\sigma_{\infty}',\sigma_1',\tau')$ with $\sigma_{\infty}=\sigma_{\infty}'=(2n, 2n-1, \ldots, 1)$ and such that $\Sigma'=\gamma^{-1} \Sigma \gamma$ for some $\gamma\in S_{2n}$, then $\gamma$ must be equal to a power of $\sigma_{\infty}$.
We will consider this conjugation from a geometric point of view; namely, if we consider the regular $2n$-gon with vertices labelled by $1,2,\ldots, 2n$, then conjugating by a permutation of the form $\sigma_{\infty}^\ell$ with $\ell\in\{1,\ldots,2n\}$ corresponds to a rotation of angle $\frac{\pi}{n} \ell$.

Let us denote by $p:=\frac{k+h}{2}$ and $q:=\frac{2n-h+k}{2}$ the two points different from $2n$ fixed by $\sigma_1\tau$. Then, if we conjugate the set $\{2n, q, p\}$ by a power of $\sigma_{\infty}$, it corresponds to a rotation of the triangle of vertices $\{2n, q, p\}$. As we want $\Sigma'$ to be special, then $\sigma_1'\tau$ has to fix $2n$, hence we are interested in the rotations which send one of the vertices to $2n$. This means that we can conjugate only by $\sigma_{\infty}^{2n-p}$ or by $\sigma_{\infty}^{2n-q}$. In the first case, the set of fixed points $\{2n, p, q\}$ is sent to $\{2n, 2n-p, q-p\}$, while, in the second case, it is sent to $\{2n, 2n-(q-p), 2n-q\}$. 


We have the following result:

\begin{proposition}
Let $\Sigma$ be a special $4$-tuple $(\sigma_{\infty}, \sigma_0,\sigma_1,\tau)$ such that $\sigma_1\tau$
contains the $3$-cycle $(2n-h,h,k)$. Then, the conjugacy class of $\Sigma$ contains exactly three special $4$-tuples.
\end{proposition}

\begin{proof}
Since the $4$-tuple $\Sigma=(\sigma_{\infty}, \sigma_0,\sigma_1,\tau)$ is special, we have that $\sigma_{\infty}=(2n, 2n-1, \ldots, 1)$ and $\sigma_1\tau$ fixes $2n$; moreover, by assumption it contains the $3$-cycle $(2n-h,h,k)$. Then, as proved before, $\sigma_0$ and  $\sigma_1\tau$ will have the shape \eqref{3-cycle}.
As discussed before, if $\Sigma'\neq \Sigma$ is a special $4$-tuple conjugated to $\Sigma$ then $\Sigma'=\gamma^{-1}\Sigma \gamma$ with $\gamma\in \left \{ \sigma_{\infty}^{2n-\frac{k+h}{2}}, \sigma_{\infty}^{2n-\frac{2n-h+k}{2}} \right \}$. To prove the assertion, we have to show that no $4$-tuple is fixed by such a conjugation. To do this, we consider two cases.

Assume first that $3 \mid n$ and that the $3$-cycle contained in $\sigma_1\tau$ is $\left (\frac{n}{3}, \frac{5}{3}n, n\right )$, i.e. $(h,k)= (\frac{n}{3}, n)$. By the description above, $\tau \in \left \{ \left (\frac{n}{3}, n \right ), \left (n, \frac{5}{3}n\right ), \left (\frac{n}{3}, \frac{5}{3} n \right )  \right \}$ and $\sigma_1\tau$ fixes $ \left \{\frac{2n}{3}, \frac{4n}{3},2n \right \}$ (which are the vertices of an equilateral triangle). Let us assume without loss of generality that $\tau=\left (\frac{n}{3}, n \right )$.
If we conjugate $\Sigma$ by $\gamma\in \left \{\sigma_{\infty}^{\frac{2n}{3}},\sigma_{\infty}^{\frac{4n}{3}} \right \}$, then $\gamma^{-1}\sigma_0 \gamma=\sigma_0$ and $\gamma^{-1}(\sigma_1\tau) \gamma=\sigma_1\tau$; hence
$\Sigma$ will be conjugated to $\Sigma'=\left (\sigma_\infty, \sigma_0, \sigma_1', \left (\frac{5}{3}n, n \right ) \right )$ and $\Sigma''=(\sigma_\infty, \sigma_0, \sigma_1'', \left (\frac{n}{3}, \frac{5}{3} n \right ))$.

Assume now $(h,k)\neq (\frac{n}{3}, n)$. As the conjugation preserves the disjoint cycle structure, $\sigma_1\tau$ will contain a $3$-cycle and by construction one of the indexes fixed by $\sigma_1\tau$ is $2n$, so the conjugated $4$-tuples will have the shape \eqref{shape_3_cycle}. We have only to compute the images of the $3$-cycles. A direct computation shows that, if we conjugate by $\sigma_{\infty}^{2n-\frac{k+h}{2}}$ we have $\Sigma'$ with $(h',k')=\left ( \frac{k-h}{2}, 2n- \frac{3h+k}{2}\right )$ while if we conjugate by $\sigma_{\infty}^{2n-\frac{2n-h-k}{2}}$ we have $\Sigma'$ with $(h',k')=\left ( n- \frac{h+k}{2}, n+\frac{3h-k}{2}\right )$, as wanted.
\end{proof}

Using these two propositions we can count the different conjugacy classes of $4$-tuples such that $\sigma_1\tau$ contains a $3$-cycle.

Let us consider a special $4$-tuple $\Sigma=(\sigma_0,\sigma_{\infty},\sigma_1,\tau_1)$ such that $\sigma_1\tau_1$ contains the $3$-cycle $(2n-h,h,k)$.
 By \eqref{shape_3_cycle} we have $h\equiv k \pmod 2$; moreover, every choice of the couple $(h,k)$ gives rise to three different choices of the couple $(\sigma_1,\tau)$, and by the previous proposition, each conjugacy class contains exactly three special $4$-tuples.
Consequently, the number of conjugacy classes is equal to the number of different choices of the couple $(h,k)$, i.e. to the number of couples $\{h,k\}$ such that $1\le h\le n-2$, $h+1<k<2n-h-1$ and $h\equiv k$ (mod $2$), hence
\[
 \# \textrm{ conjugacy classes}= \sum_{h=1}^{n-2} \frac{2n-2(h+1)}{2}=\sum_{h=1}^{n-2} (n-h-1)= \sum_{j=1}^{n-2} j= \frac{(n-1)(n-2)}{2}.
\]

\subsection{The 4-cycle case} \label{4-cycle}

Let $\Sigma=(\sigma_0,\sigma_{\infty},\sigma_1,\tau)$ be a special $4$-tuple, i.e. with $\sigma_{\infty}=(2n, \ldots, 1)$ and such that $\sigma_1$ and $\tau$
fix both $2n$, and assume that $\sigma_1\tau$ is a product of $n-4$ disjoint transpositions and a disjoint $4$-cycle. 

As in the previous section, our first result describes the shape of a special $4$-tuple of this kind.

\begin{proposition}
Let $\Sigma=(\sigma_0,\sigma_{\infty}, \sigma_1, \tau)$ be a special $4$-tuple such that $\sigma_1\tau$ contains a $4$-cycle. Then, there exist $1\le h< k_1 <k_2<2n-h$ with $h\equiv k_1\equiv k_2 \pmod 2$ such that 
\begin{align} \label{shape_4_cycle}
\sigma_0 &=\prod_{i=1}^h(i,2n+1-i) \prod_{j=1}^{\frac{k_1-h}{2}}(h+j, k_1+1-j)
\prod_{t=1}^{\frac{k_2-k_1}{2}} (k_1+t, k_2+1-t) \nonumber \\
&\prod_{v=1}^{\frac{2n-h-k_2}{2}}(k_2+v, 2n-h+1-v); \\
\sigma_1\tau &=\prod_{i=1}^{h-1}(i,2n-i)\prod_{j=1}^{\frac{k_1-h}{2}-1} (h+j, k_1-j) \prod_{t=1}^{\frac{k_2-k_1}{2}-1} (k_1+t, k_2-t) \nonumber \\
&\prod_{v=1}^{\frac{2n-h-k_2}{2}-1}(k_2+v,2n-h-v)\ (2n-h, h, k_1, k_2). \nonumber
\end{align}
In particular $\sigma_1\tau$ fixes four indexes: $2n, \frac{k_1+h}{2}$, $\frac{k_1+k_2}{2}$  and $\frac{2n-h+k_2}{2}$.
Moreover, there are two different $\Sigma$ satisfying \eqref{shape_4_cycle} corresponding to different choices of $\tau$, i.e. $\tau\in \{(h,k_2), (k_1, 2n-h)\}$.
\end{proposition}

\begin{proof}
Let $\Sigma=(\sigma_0,\sigma_{\infty},\sigma_1,\tau)$ be a special $4$-tuple and assume that $\sigma_1\tau$ is a product of $n-4$ disjoint transpositions and a disjoint $4$-cycle. In this case, $\sigma_1\tau$ fixes $4$ indexes, and one of them is $2n$.
Then, as $\sigma_0\sigma_{\infty}\sigma_1\tau=\id$, we have that $\sigma_0(2n)=1$. As $\sigma_0$ is composed only by transpositions, we have also that $\sigma_0(1)=2n$, hence $\sigma_1\tau(2n-1)=1$. Now, we have two choices: either $\sigma_1\tau(1)=2n-1$ (so $(2n-1,1)$ is one of the disjoint transpositions in the decomposition of $\sigma_1\tau$), or $\sigma_1\tau(1)=k\neq 2n-1$ (giving rise to the $4$-cycle in the product). 
We go on as in the disjoint case described in Subsection \ref{disjoint} until we reach the situation in which there exist $1\le h\le n-2$, $h+1\le k_1\le 2n-h-3$, $k_1+1\le k_2\le 2n-h-1$ such that $\sigma_0$ contains the disjoint transpositions $(1, 2n),(2,2n-1)\cdots (h,2n-h+1)$ and $\sigma_1\tau$ contains the disjoint transpositions $(1,2n-1),\ldots, (2n-h+1,h-1)$ and the 4-cycle $(2n-h,h,k_1, k_2)$. 
 It is easy to see that $k_1\neq h+1, 2n-h-3$ and $k_2\neq k_1+1, 2n-h-1$.
For example, let us assume by contradiction that $k_1=h+1$; then, as we have that $\sigma_0\sigma_{\infty}\sigma_1\tau_1=\id$, we would have that $\sigma_0$ fixes $h+1$, which is impossible because $\sigma_0$ has no fixed points. For the other cases we can argue similarly. So, $h+2 \le k_1 \le 2n-h-4$ and $k_1+2\le k_2\le 2n-h-2$, which implies that $1\le h\le n-3$. \\

Now, if $\sigma_1\tau(h)=k_1$, then by $\sigma_0\sigma_{\infty}\sigma_1\tau=\id$ we have $\sigma_0(k_1)=h+1$ and so $\sigma_0(h+1)=k_1$. This means that $\sigma_1\tau(h+1)=k_1-1$ and $\sigma_1\tau(k_1-1)=h+1$ and so on. This process ends when one reaches the conditions $\sigma_0\left ( \frac{k_1+h}{2}\right)=\frac{k_1+h}{2}+1$ and viceversa, $\sigma_1\tau\left (\frac{k_1+h}{2}-1\right )=\frac{k_1+h}{2}+1$ and viceversa and $\frac{k_1+h}{2}$ is fixed by $\sigma_1\tau$. Notice that also in this case this gives the additional condition that $h\equiv k_1 \pmod 2$. Here we also see that $k_2$ cannot lie between $h$ and $k_1$.

On the other hand, if $\sigma_1\tau$ contains the 4-cycle $(2n-h,h,k_1,k_2)$, we also have that $\sigma_1\tau(k_1)=k_2$ and $\sigma_1\tau(k_2)=2n-h$, giving other conditions to satisfy. Indeed, by $\sigma_0\sigma_{\infty}\sigma_1\tau=\id$, we have $\sigma_0(2n-h)=k_2+1$ and $\sigma_0(k_2+1)=2n-h$. This implies that $\sigma_1\tau(2n-h-1)=k_2+1$ and viceversa, and so on. This process ends when one reaches the conditions $\sigma_0\left (\frac{2n-h+k_2}{2}\right)=\frac{2n-h+k_2}{2}+1$ and viceversa, $\sigma_1 \tau\left ( \frac{2n-h+k_2}{2}-1\right)=\frac{2n-h+k_2}{2}+1$ and $\frac{2n-h+k_2}{2}$ is fixed by $\sigma_1\tau$. With the same argument, as $\sigma_1\tau(k_1)=k_2$, by $\sigma_0\sigma_{\infty}\sigma_1\tau=\id$, we have $\sigma_0(k_1+1)=k_2$ and $\sigma_0(k_2)=k_1+1$. This means that $\sigma_1\tau(k_1+1)=k_2-1$ and viceversa, and so on. This process ends when one reaches the conditions $\sigma_0\left ( \frac{k_1+k_2}{2}\right)=\frac{k_2+k_2}{2}+1$ and viceversa, $\sigma_1 \tau\left ( \frac{k_1+k_2}{2}-1\right)=\frac{k_1+k_2}{2}+1$ and $\frac{k_1+k_2}{2}$ is fixed by $\sigma_1\tau$. This gives, as before, the additional condition that $k_1\equiv k_2 \pmod 2$. This proves that $\sigma_0$ and $\sigma_1\tau$ have the shape \eqref{shape_4_cycle}, proving the first part of the proposition.

In particular $\sigma_1\tau$ fixes 4 indexes, i.e. $\frac{k_1+h}{2}$, $\frac{k_1+k_2}{2}$  and $\frac{2n-h+k_2}{2}$ and $2n$. Let us denote by 
\[
\beta_{hk_1k_2}=: \sigma_1\tau (2n-h, h, k_1, k_2)^{-1}, 
\]
i.e. the product of the disjoint transpositions appearing in the decomposition of $\sigma_1\tau$. Then, we have two different choices for the couple $(\sigma_1,\tau)$ giving the same product, namely $\sigma_1=\beta_{hk_1k_2} (2n-h, h)(k_1,k_2)$ and $\tau=(2n-h,k_1)$ or $\sigma_1=\beta_{hk_1k_2} (h,k_1)(2n-h, k_2)$ and $\tau=(h, k_2)$. This concludes the proof.
\end{proof}

As in the previous section, we want to count how many special $4$-tuples such that $\sigma_1\tau$ contains a $4$-cycle there are in each conjugacy class.

Notice that, if we have two special $4$-tuples $\Sigma=(\sigma_0,\sigma_{\infty},\sigma_1,\tau_1)$ and $\Sigma'=(\sigma_0',\sigma_{\infty}',\sigma_1',\tau_1')$ with $\sigma_{\infty}=\sigma_{\infty}'=(2n, 2n-1, \ldots, 1)$ and such that $\Sigma'=\gamma^{-1} \Sigma \gamma$ for some $\gamma\in S_{2n}$, then $\gamma$ must be a power of $\sigma_{\infty}$.

Let us call $p:= \frac{k_1+h}{2},q:=\frac{k_1+k_2}{2}$ and $r:=\frac{2n-h+k_2}{2}$ the three indexes other than $2n$ fixed by $\sigma_1\tau$. Since we are interested in conjugations by permutations of the form $\sigma_{\infty}^\ell$ that send the set $\{p,q,r,2n\}$ into a set containing $\{2n\}$, the only admissible $\gamma$ will be $\sigma_{\infty}^{2n-p}$, $\sigma_{\infty}^{2n-q}$ or $\sigma_{\infty}^{2n-r}$. In particular, in the first case the set $\{2n, p, q, r\}$ is sent to $\{2n, 2n-p, q-p, r-p\}$, in the second case to $\{2n, 2n-(q-p), r-q, 2n-q\}$ and, in the last case, to $\{2n-r, 2n-(r-p), 2n-(r-q), 2n\}$. 

The following proposition describes how many special $4$-tuples we have in every conjugacy class of a $4$-tuple; in this case, this depends on the configuration of $\{p,q,r,2n\}$.

\begin{proposition} \label{prop:4-cycle}
Let $\Sigma=(\sigma_0, \sigma_{\infty},\sigma_1,\tau_1)$ be a special $4$-tuple such that $\sigma_1\tau$ contains the $4$-cycle $(h,k_1,k_2,2n-h)$; then we have two possibilities.
\begin{itemize}
\item If $n$ is even and the $4$-cycle is of the form $(h,n-h,n+h,2n-h)$, then the conjugacy class of $\Sigma$ contains only two special $4$-tuples;
\item if not, then the conjugacy class of $\Sigma$ contains exactly four special $4$-tuples.
\end{itemize}
\end{proposition} 

We point out that, looking at the symmetry of the problem, the number of special $4$-tuples contained in the conjugacy class of $\Sigma$ depends on the configuration of the points fixed by $\sigma_1\tau$. Indeed, in the first case, the set of fixed points is exactly $\left \{\frac{n}{2}, n, \frac{3}{2}n, 2n \right \}$, which are the vertices of a square.

\begin{proof}
Since $4$-tuple $\Sigma=(\sigma_{\infty}, \sigma_0,\sigma_1,\tau)$ is special, we have that $\sigma_{\infty}=(2n, 2n-1, \ldots, 1)$ and $\sigma_1\tau$ fixes $2n$; moreover, by assumption it contains the $4$-cycle $(h,k_1,k_2,2n-h)$. Then, as proved before, $\sigma_0$ and  $\sigma_1\tau$ will have the shape \eqref{4-cycle}.
As discussed before, if $\Sigma'\neq \Sigma $ is a special $4$-tuple conjugated to $\Sigma$ then $\Sigma'=\gamma^{-1}\Sigma \gamma$ with $\gamma\in \left \{ \sigma_{\infty}^{2n-\frac{h+k_1}{2}}, \sigma_{\infty}^{2n-\frac{k_1+k_2}{2}}, \sigma_{\infty}^{2n-\frac{2n-h+k_2}{2}} \right \}$. We distinguish two cases.

Suppose that $n$ is even and the $4$-cycle is of the form $(h,n-h,n+h,2n-h)$\footnote{Note that this configuration of the $4$-cycle is not possible if $n$ is odd since $h$, $n-h$ and $n+h$ would not have the same parity.}; then, the points fixed by $\sigma_1\tau$ are exactly $\left \{\frac{n}{2}, n, \frac{3}{2}n, 2n \right \}$. As discussed before, the other special $4$-tuples contained in the conjugacy class of $\Sigma$ will be of the form $\gamma^{-1}\Sigma \gamma$, with $\gamma\in \left \{ \sigma_{\infty}^{\frac{n}{2}}, \sigma_{\infty}^
{n}, \sigma_{\infty}^{\frac{3}{2}n} \right \}$. But for this special configuration, a direct computation shows that $\sigma_{\infty}^{-n}\Sigma \sigma_{\infty}^n=\Sigma$, and $\sigma_{\infty}^{-n/2}\Sigma \sigma_{\infty}^{n/2}=\sigma_{\infty}^{n/2}\Sigma \sigma_{\infty}^{-n/2}\neq \Sigma$, hence there are only two special $4$-tuples contained in the conjugacy class of $\Sigma$.

Suppose now that the $4$-cycle contained in $\sigma_1\tau$ is not of the previous shape; we have that either $k_1\neq n-h$ or $k_2\neq n+h$. In this case, a direct computation shows that if we conjugate $\Sigma$ by some $\gamma\in \left \{ \sigma_{\infty}^{2n-\frac{h+k_1}{2}}, \sigma_{\infty}^{2n-\frac{k_1+k_2}{2}}, \sigma_{\infty}^{2n-\frac{2n-h+k_2}{2}} \right \}$, then $\sigma_0$ is not fixed, hence every conjugation gives rise to a special $4$-tuple in the conjugacy class of $\Sigma$, which concludes the proof.
\end{proof}

Using the proposition, we can count the number of different conjugacy classes such that $\sigma_1\tau$ contains a $4$-cycle; in this case we have to distinguish the case $n$ odd and $n$ even.

Suppose first $n$ odd; in this case, by Proposition \ref{prop:4-cycle}, every conjugacy class contains exactly four special $4$-tuples. Moreover, a special $4$-tuple containing a $4$-cycle $(h,k_1,k_2, 2n-h)$ with $h<k_1<k_2$ will have the shape \eqref{4-cycle}, and there are two different choices of the couple $(\sigma_1, \tau)$ which gives the same permutation $\sigma_1\tau$.

Putting all together, we have to count the number of ordered $3$-tuples $h<k_1<k_2$ with $h\equiv k_1 \equiv k_2 \pmod 2$ and such that $1\le h\le n-3$, $h+2\le k_1\le 2n-4-h$ and $k_1+2\le k_2\le 2n-2-h$, which are equal to 
\[
C_1= \sum_{h=1}^{n-3} \left [ \sum_{\tiny{\begin{matrix} k=h+2  \\ k\equiv h \pmod 2 \end{matrix}}}^{2n-4-h} \left ( n-1 -\frac{k+h}{2} \right ) \right ].
\]
Then we have that, if $n$ is odd, the number of conjugacy classes of a $4$-tuple such that $\sigma_1\tau$ contains a $4$-cycle is 
\[
\# \mbox{ conjugacy classes}= 2\cdot \frac{C_1}{4}= \frac{C_1}{2}.
\]

Suppose now $n$ even; in this case, we have to distinguish the cases in which the $4$-cycle is of the special form $(h,n-h,n+h,2n)$, which are $C_2=\frac{n}{2}-1$. Recall that, for every configuration of the $4$-cycle, we have two choices of the couple $(\sigma_1, \tau)$ giving rise to the same product. For this special configuration of the $4$-cycle, the conjugacy class contains indeed only two special $4$-tuples. The number of conjugacy classes in this case is given by
\[
\# \mbox{ conjugacy classes}= 2\cdot \frac{C_1-C_2}{4}+ 2\cdot \frac{C_2}{2}= \frac{C_1+C_2}{2}.
\]


\bibliographystyle{amsalpha}
\bibliography{bibliography}

\end{document}